\documentclass[12pt]{article}
\usepackage{amsmath}
\usepackage{amsfonts}
\usepackage{amssymb}
\usepackage{authblk}
\usepackage{mathrsfs}
\usepackage{amsthm}
\usepackage{subcaption}
\usepackage{graphicx}
\usepackage{xcolor}
\usepackage[normalem]{ulem}
\usepackage{hyperref}
\hypersetup{
    colorlinks=true,
    citecolor=green!50!black,
    linkcolor=red!80!black,
    filecolor=blue,
    urlcolor=blue!70!black,
}

\newtheorem{proposition}{Proposition}
\newtheorem{theorem}{Theorem}

\newtheorem{definition}{Definition}

\theoremstyle{definition}

\DeclareMathOperator{\diver}{div}

\DeclareMathOperator{\curl}{curl}
\DeclareMathOperator{\Ric}{Ric}

\DeclareMathOperator{\End}{End}

\newcommand{\ad}{\mathrm{ad}}

\newcommand{\adstar}{\ad^{\star}}

\newcommand{\id}{\mathrm{id}}

\newcommand{\Diff}{\mathrm{Diff}}
\newcommand{\Diffmu}{\Diff_{\mu}}
\newcommand{\sgrad}{\nabla^{\perp}}
\newcommand{\sgradbar}{\tilde\nabla^{\perp}}

\newcommand{\Laplacian}{\Delta}

\newcommand{\adc}{\mathbf{\mathcal{C}}}
\newcommand{\dbar}{\mathbf{\mathcal{D}}}
\newcommand{\hbarN}{\hbar}
\newcommand{\sigmamat}{B}

\newcommand{\arxivonly}[1]{#1}
\newcommand{\journalonly}[1]{}

\newcommand{\new}[1]{#1}
\newcommand{\old}[1]{}

\title{Zeitlin's model for axisymmetric 3-D Euler equations}
\author[1]{Klas Modin}
\author[2]{Stephen C. Preston}
\affil[1]{Mathematical Sciences, Chalmers and University of Gothenburg, Gothenburg, SE-412~96, Sweden \journalonly{${}^1$corresponding author, \texttt{klas.modin@chalmers.se}}}
\affil[2]{Department of Mathematics, CUNY Brooklyn College, New York, NY 10468 
and CUNY Graduate Center, New York, NY 10016, USA}

\date{}

\begin{document}

\maketitle 

\begin{center}
   \vspace{-3ex}
   \it To the memory of Vladimir Zeitlin.
\end{center}

\begin{abstract}
Zeitlin's model is a spatial discretization for the 2-D Euler equations on the flat 2-torus or the 2-sphere. Contrary to other discretizations, it preserves the underlying geometric structure, namely that the Euler equations describe Riemannian geodesics on a Lie group. Here we show how to extend Zeitlin's approach to the axisymmetric Euler equations on the 3-sphere. It is the first discretization of the 3-D Euler equations that fully preserves the geometric structure\new{, albeit restricted to axisymmetric solutions}. Thus, this finite-dimensional model admits Riemannian curvature and Jacobi equations, which are discussed. 
\\[1ex]
\textbf{Keywords:} axisymmetric Euler equations, Zeitlin's model, sectional curvature, Euler-Arnold equations, Abelian extension
\\[1ex]
\textbf{MSC2020:} 35Q31, 53D50, 76M60, 76B47, 53D25
\end{abstract}

\arxivonly{
\tableofcontents
}

\section{Introduction}

Euler's~\cite{Eu1757} equations for an ideal fluid are the second-oldest 
partial differential equations ever written down.\footnote{Only the wave equation is older.}
They are widely studied, but many of their aspects remain abstruse.
It is therefore important to find finite-dimensional models that preserve as much of their structure as possible, both 
for theoretical purposes and for numerical simulation.
In particular, the equations describe geodesics in the group of volume preserving diffeomorphisms of a domain under a right-invariant metric corresponding to kinetic energy (as described by Arnold~\cite{Ar1966}).
A candidate for modeling this structure is a finite-dimensional Lie group with a right-invariant Riemannian metric.
For the 2-D Euler equations, an effective family of such models was given by Zeitlin, first on the flat torus~\cite{Ze1991} and then on the sphere~\cite{Ze2004}.
The latter uses spherical harmonics and their relation to representation theory for $SO(3)$, such that the approximating groups are the special unitary groups $SU(n)$ for positive integers~$n$.
In consequence, these models respect the $SO(3)$ symmetry of the sphere, which implies better convergence and less Gibbs phenomena than the corresponding torus models (which fail to preserve the translational symmetry).
Zeitlin's model has been exploited to study the long-time behavior of spherical solutions, by the first author and others (see, e.g., \cite{MoVi2020,FrCaCiGe2024} and references therein).

In this note we show how to extend Zeitlin's model to\old{the three-dimensional (3-D) case} \new{
axisymmetric solutions of the three-dimensional (3-D) Euler equations on the $3$-sphere.} 
Indeed, on the $3$-sphere, the Hopf vector field generates a family of isometries, and its flow lines are all circles of the same length. 
The quotient by this flow is the well-known Hopf fibration onto the $2$-sphere. Solutions of the 3-D Euler equation that commute with this Hopf flow are called axisymmetric by analogy with the rotation field in $3$-space, and the 3-D axisymmetric Euler equation reduces to a pair of equations on the $2$-sphere~\cite{LiMiPr2022}. 
These equations can be approximated by the Zeitlin model in the same way as in the two-dimensional case, and we end up with a model for axisymmetric 3-D Euler equations on $3$-spheres in terms of a finite dimensional space $\mathfrak{su}(n)\times\mathfrak{u}(n)$ equipped with a twisted Lie algebra product. 
We will describe this Lie algebra structure and some aspects of its geometry, along with results of 3-D numerical simulations obtained by the same techniques as in the 2-D case. 

\new{
From a physics standpoint, Euler's equations on the 3-sphere can be seen as a model for spin orientation waves, via the identification $S^3\simeq \mathrm{SU}(2)$.
However, the key aspect of the axisymmetric 3-D Euler equations emanate from analysis.
Indeed, whereas the 2-D equations are globally well-posed, the 3-D Euler equations are likely not, and the obstacle in the analysis appears precisely in the restriction to axisymmetric solutions with swirl.} 
\new{In particular, the mechanism for blow-up discovered by Luo and Hou~\cite{LuHo2014} and developed by Elgindi~\cite{El2021} builds on axisymmetric solutions, as does the more recent result by Chen and Hou~\cite{ChHo2023}.}

\new{The extended Zeitlin model developed here thus provides a geometrically consistent approach for numerical studies of the qualitative differences in dynamical behavior between Euler's equations in 2-D and 3-D.
A first indication is given in section~\ref{sec:numerics} below, where two numerical experiments suggests a generic faster-than-exponential growth of vorticity, which cannot occur in 2-D. This conforms with the general expectation that axisymmetric ideal fluid flow with nonzero swirl already captures all the worst possible behavior of the equations in the fully general case. The fact that breakdown mechanisms appear to be a \emph{local} phenomenon indicates that the distinction between $S^3$ and e.g., $\mathbb{T}^3$ is not important at this level, and working on $S^3$ leads to simpler formulas.}

\subsection*{Acknowledgements}

We would like to thank Cornelia Vizman for useful discussions vital to our results \new{and Michael Roop for pointing out the general form of the Casimir functions}.
\new{We also acknowledge the anonymous reviewers for helpful suggestions.}
The work was supported by the Swedish Research Council (grant number 2022-03453) and the Knut and Alice Wallenberg Foundation (grant number WAF2019.0201).
The computations were enabled by resources provided by Chalmers e-Commons at Chalmers.

\section{Background}

\subsection{General aspects}
For the material in this portion, we refer to the monograph by Arnold and Khesin~\cite{ArKh1998}. 
Let $M$ be a compact simply connected Riemannian manifold without boundary, either $S^2$ or $S^3$ with the usual round metric
of constant curvature $1$.  
Euler's equations for the velocity field $u(t,x)$ of an ideal fluid on $M$ take the form 
$$ \dot u + \nabla_uu = -\nabla p, \qquad \diver{u}=0, $$
where $\dot u$ denotes derivative with respect to time and the pressure $p$ is determined implicitly by the volume preserving constraint via $\Delta p = -\diver{(\nabla_uu)}$. 
Eliminating the pressure by taking the curl gives two versions of the equation for the vorticity $\omega=\curl{u}$, depending on the dimension:
\begin{align} 
\dot\omega + u\cdot \nabla \omega &= 0, \qquad \omega=\curl{u}, \text{ a function in two dimensions;} \label{2Dvorticityvelocity} \\
\dot\omega + [u,\omega] &= 0, \qquad \omega=\curl{u}, \text{ a vector field in three dimensions.} \label{3Dvorticityvelocity}
\end{align}

Since $M$ is simply connected, the curl $\omega$ completely determines the divergence-free field $u$ 
via a Biot-Savart operator. In two dimensions we can write $u=\sgrad \psi$ where $\psi$ is a stream function 
and $u$ is defined in terms of the area $2$-form $\mu$ by the condition that $\iota_u\alpha \, \mu = -\alpha \wedge d\psi$ 
for every $1$-form $\alpha$ on $M$; with this convention\footnote{Many authors choose the opposite convention for the 
stream function, which will flip the sign in all equations but otherwise does not matter.} the vorticity becomes $\omega = \Laplacian \psi$
and we have $u\cdot \nabla\omega = \{\psi, \omega\}$ in terms of the Poisson bracket defined by $d\psi\wedge d\omega = \{\psi,\omega\} \mu$,
so the 2-D Euler equation becomes 
\begin{equation}\label{2Dvorticitystream} 
\dot\omega + \{\psi, \omega\} = 0, \qquad \Delta \psi=\omega.
\end{equation}

The flow of the time-dependent velocity field is denoted by $\gamma$, satisfying 
$$ \dot\gamma(t,x) = u\big(t,\gamma(t,x)\big), \qquad \gamma(0,x)=x,$$
and the volume preserving condition $\det{D_x\gamma}\equiv 1$. 
The group of such volume preserving diffeomorphisms is denoted $\Diffmu(M)$. 
In terms of the flow $\gamma$, the vorticity equations \eqref{2Dvorticityvelocity}--\eqref{3Dvorticityvelocity} 
can be solved to give the vorticity transport laws 
$$ \omega(t,\gamma(t,x)) = \omega_0(x) \quad \text{(2-D),} \qquad \omega(t,\gamma(t,x)) = D_x\gamma(t,x)\omega_0(x) \quad \text{(3-D)}.$$
These correspond to the left action of $\gamma(t)\in\Diffmu(M)$ on the initial vorticity configuration~$\omega_0$.

If $G$ is a group (finite- or infinite-dimensional) with a right-invariant metric $\langle \cdot, \cdot\rangle$, then the equation for a geodesic $\gamma(t)\in G$ starting at the identity can be written as the coupled system 
\begin{equation}\label{eulerarnold}
\dot\gamma(t) = u(t) \gamma(t), \qquad \dot u(t) + \adstar_{u(t)}u(t) = 0, \qquad \gamma(0)=\id, \quad u(0)=u_0\in \mathfrak{g},
\end{equation}
where $\adstar$ is the operator defined by 
\begin{equation}\label{adstardef}
\langle \adstar_uv,w\rangle = \langle v, \ad_uw\rangle \qquad \forall u,v,w\in \mathfrak{g}.
\end{equation}
The equation for $\gamma$ is called the flow equation, while the equation for $u$ is called the Euler-Arnold equation. 
The Euler equations correspond to $G=\Diffmu(M)$, with $\mathfrak{g}$ given by the divergence-free vector fields, and the right-invariant metric given by the $L^2$ kinetic energy 
$$ \langle u,v\rangle  = \int_M g(u,v) \, \mu_g.$$

The curvature tensor is given for vectors $u$ and $v$ by the formula
\begin{equation}\label{curvature}
\langle R(u,v)v,u\rangle = 
\tfrac{1}{4} \lvert \adstar_uv + \adstar_vu + \ad_uv\rvert^2 
- \langle \adstar_uv+\ad_uv, \ad_uv\rangle - \langle \adstar_uu, \adstar_vv\rangle,
\end{equation}
which comes from completing the square in the Arnold formula. 

The Jacobi equation is the linearization of the Euler-Arnold equation \eqref{eulerarnold}, and splits in the same way: a Jacobi field $J(t)=y(t)\gamma(t)$ along a geodesic $\gamma$ satisfies the equation 
\begin{equation}\label{jacobisplit}
\dot y(t) - \ad_{u(t)}y(t) = z(t), \qquad \dot z(t) + \adstar_{u(t)}z(t) + \adstar_{z(t)}u(t) = 0.
\end{equation}
Conjugate points along geodesics occur when there is a solution of this equation with $y(0)=0$ and $y(T)=0$ for some $T>0$. 
See \cite{KhLeMiPr2011b} for a survey of results about curvatures and conjugate points on $\Diffmu(M)$.

\subsection{Zeitlin's model on the 2-sphere}\label{zeitlinsection}

Zeitlin's model originates from quantization theory developed by Hoppe~\cite{Ho1989}.
The idea is to replace the Poisson algebra of smooth functions on a symplectic manifold $M$ with a Lie algebra of skew-Hermitian operators in such a way that (i) the operator eigenvalues correspond to the function values and (ii) the operator commutator corresponds to the Poisson bracket.
If the manifold $M$ is compact and quantizable (cf.~\cite{ChPo2018}), the operators can be taken as $\mathfrak{u}(n)$ matrices in such that the classical limit $\hbar\to 0$ corresponds to $n\to\infty$.
The final ingredient is a quantum version $\Delta_n\colon\mathfrak{u}(n)\to\mathfrak{u}(n)$ of the Laplacian $\Delta\colon C^\infty(M)\to C^\infty(M)$. 
Zeitlin's model is then given by
\begin{equation}\label{zeitlinbasic}
   \dot W + \frac{1}{\hbar}[P, W]=0, \qquad \Delta_n P = W ,
\end{equation}
which yields a spatial discretization of the vorticity equation~\eqref{2Dvorticitystream}.

Hoppe and Yau~\cite{HoYa1998} constructed quantization for $M=S^2$ from representation theory for~$\mathfrak{so}(3)$.
Indeed, for integer $n$, let $s=\frac{n-1}{2}$ (the ``spin'' number). 
Then construct three matrices $S_1,S_2,S_3\in \mathfrak{u}(n)$, with indices labeled from $-s$ to $s$ (instead of $1$ to $n$), such that 
\begin{itemize}
\item $S_1$ is purely imaginary and symmetric, whose only nonzero entries above the diagonal are $a_{j,j+1} = \tfrac{i}{2}\sqrt{s(s+1)-j(j+1)}$;
\item $S_2$ is purely real and antisymmetric, whose only nonzero entries above the diagonal are are  $b_{j,j+1} = \tfrac{1}{2}\sqrt{s(s+1)-j(j+1)}$;
\item $S_3$ is purely imaginary and diagonal, with diagonal entries $c_{jj} = i j$.
\end{itemize}
For example, when $n=3$ we get $s=1$ and 
$$ 
S_1 = \frac{1}{\sqrt{2}} \left( \begin{matrix}
0 & i & 0 \\
i & 0 & i \\
0 & i & 0\end{matrix}\right)
\qquad 
S_2 = \frac{1}{\sqrt{2}} \left( \begin{matrix}
0 & 1 & 0 \\
-1 & 0 & 1 \\
0 & -1 & 0\end{matrix}\right), \qquad
S_3 = \left( \begin{matrix} 
-i & 0 & 0 \\
0 & 0 & 0 \\
0 & 0 & i \end{matrix}\right).$$
Meanwhile, for $n=4$ we get $s=\tfrac{3}{2}$ and 
$$ S_1 = \left( \begin{matrix}
0 & \tfrac{\sqrt{3}}{2} i & 0 & 0 \\
\tfrac{\sqrt{3}}{2} i & 0 & i & 0 \\
0 & i & 0 & \tfrac{\sqrt{3}}{2}i \\
0 & 0 & \tfrac{\sqrt{3}}{2} i & 0 \end{matrix}\right), 
\quad 
S_2 = \left( \begin{matrix}
0 & \tfrac{\sqrt{3}}{2}  & 0 & 0 \\
-\tfrac{\sqrt{3}}{2}  & 0 & 1 & 0 \\
0 & -1 & 0 & \tfrac{\sqrt{3}}{2} \\
0 & 0 & -\tfrac{\sqrt{3}}{2}  & 0 \end{matrix}\right), 
\quad 
S_3 = \left( \begin{matrix} 
-\tfrac{3i}{2} & 0 & 0 & 0 \\
0 & -\tfrac{i}{2} & 0 & 0 \\
0 & 0 & \tfrac{i}{2} & 0 \\
0 & 0 & 0 & \tfrac{3i}{2}\end{matrix}\right).
$$
The matrices $S_1,S_2,S_3$ provide an irreducible representation of the Lie algebra $\mathfrak{so}(3)$ on $\mathbb{C}^n$, so they fulfil the commutation relations 
$$ [S_1,S_2]=S_3, \qquad [S_2,S_3] = S_1, \qquad [S_3,S_1]=S_2 .$$
In turn, they induce a representation on $\mathfrak{u}(n)$ via $\ad_{S_1},\ad_{S_2},\ad_{S_3}$, which, via the Peter--Weyl theorem, decomposes into odd-dimensional, irreducible $\mathfrak{so}(3)$-representations
\begin{equation*}
   \mathfrak{u}(n) = V_0 \oplus V_1 \oplus \cdots \oplus V_{n-1}, \qquad \dim(V_{\ell}) = 2\ell+1.   
\end{equation*}
By decomposing each $V_\ell$ according to its weights $m=0,\ldots,\ell$ we then obtain a map between spherical harmonics $Y_{\ell}^m$ and a matrix basis $T_{\ell}^m \in V_\ell$, which yields the quantization as the representation morphism $\mathcal T_n\colon C^\infty(S^2)\to \mathfrak{u}(n)$.
The scaled matrices $X_\alpha = \hbar S_\alpha$ for $\hbar = 2/\sqrt{n^2-1}$ correspond to the Cartesian coordinate functions $x_\alpha \in S^2$, whereas the scaled commutator $\frac{1}{\hbar}[\cdot,\cdot]$ converges in $L^\infty$ to the Poisson bracket $\{\cdot,\cdot\}$ as $n\to\infty$ (cf.~Charles and Polterovich~\cite{ChPo2018}).
Furthermore, the Casimir element for the representation on $\mathfrak{u}(n)$ is the Hoppe--Yau Laplacian $\Delta_n\colon \mathfrak{u}(n)\to\mathfrak{u}(n)$ given by
\begin{equation}\label{hoppeyau}
   \Delta_n =\sum_{\alpha=1}^3 \ad_{S_\alpha}^2, \qquad \text{i.e.,}\qquad \Delta_n P = \sum_{\alpha=1}^3 [S_\alpha,[S_\alpha,P]]
\end{equation}
Since the quantization operator $\mathcal T_n$ is a representation morphism it intertwines the Casimir operators, i.e., $\mathcal T_n\circ \Delta = \Delta_n\circ \mathcal T_n$.
Consequently, the Hoppe--Yau Laplacian has the right spectrum $\Delta_n\big\vert_{V_{\ell}} = -\ell(\ell+1)\mathrm{id}$.
We refer to Modin and Viviani~\cite{MoVi2024} and references therein for more details on the $S^2$ quantization, its connection to representation theory, and the corresponding Euler--Zeitlin equation~\eqref{zeitlinbasic} on $\mathfrak{u}(n)$.

Contrary to all conventional discretizations, the Euler--Zeitlin equation~\eqref{zeitlinbasic} is itself an Euler--Arnold equation, for $G=SU(n)$, $\mathfrak{g}=\mathfrak{su}(n)$, and the right-invariant metric defined at the identity by 
\begin{equation}\label{zeitlin2dmetric}
   \langle W,P\rangle = \operatorname{tr}(W\Delta_n P).
\end{equation}
Hence, there is a notion of curvature and Jacobi fields, and these notions in the finite-dimensional case approximate the corresponding objects in the infinite-dimensional case~\cite{MoPe2024}. 

\subsection{Axisymmetry on the 3-sphere}\label{sub:axisymmetry_on_spheres}

The flow of a Killing vector field $K$ on $M$ generates isometries, whose action preserves solutions of the Euler equation.
Consequently, a solution which is initially symmetric will remain so for all time (see Lichtenfelz et al.~\cite{LiMiPr2022} for details). 
At the diffeomorphism group level, the symmetry corresponds to the flow $\gamma$ commuting with the flow of $K$, while at the vector field level, it corresponds to the vanishing commutator condition $[K,u]=0$. 
On $S^2$ every Killing field is a rotation around some axis, and the condition $[K,u]=0$ is very restrictive, implying that $u$ must be a steady solution of the Euler equation. 

But in three dimensions there is more flexibility, and there is a large family of ``axisymmetric'' nonsteady solutions~\cite{LiMiPr2022}.
The restriction to such solutions is effectively a two-dimensional fluid, but with an additional source coming from the ``swirl'' $\langle u, K\rangle$. 
Explicitly, on the $3$-sphere embedded in $\mathbb{R}^4$, we may choose a basis of vector fields\footnote{These are the right-invariant vector fields of the quaternion group, and the scaling by $\tfrac{1}{2}$ is a convenience to avoid other factors of $2$ later on, but neither of these things are important in the bigger picture.}
\begin{align*}
E_1 &= \tfrac{1}{2} (-x\,\partial_w + w\,\partial_x - z\,\partial_y + y\,\partial_z) \\
E_2 &= \tfrac{1}{2} (-y\,\partial_w + z\,\partial_x + w\,\partial_y - x\,\partial_z) \\
E_3 &= \tfrac{1}{2} (-z\,\partial_w - y\,\partial_x + x\,\partial_y + w\,\partial_z),
\end{align*}
and define a Riemannian metric on $S^3$ so that these are orthonormal (corresponding to working on a $3$-sphere of radius $2$). 
The field $E_1$ is the well-known Hopf field, and the flows of each $E_i$ are $4\pi$-periodic.  

A direct calculation, using the Riemannian curl and divergence, shows that 
\begin{equation}\label{S3brackets}
[E_1,E_2] = -E_3, \qquad [E_2,E_3]=-E_1, \qquad [E_3,E_1] = -E_2, \qquad \curl{E_i} = E_i \quad \forall i.
\end{equation}
Expressing $u=\sum_{i=1}^3 u_i E_i$, the divergence is given by 
$$ \diver{u} = \sum_{i=1}^3 E_i u_i,$$
where $E_i$ acts on functions as a differential operator.
If we take $K=E_1$, the conditions $[K,u]=0$ and $\diver{u}=0$ imply the existence of functions $\tilde{\sigma}$ and $\tilde{\psi}$ such that
\begin{equation}\label{axisymmetricexplicit} 
u = \tilde{\sigma} E_1 - (E_3 \tilde{\psi}) E_2 + (E_2 \tilde{\psi}) E_3, \qquad E_1 \tilde{\sigma}\equiv 0, \quad E_1 \tilde{\psi} \equiv 0.
\end{equation}
\new{Indeed, let $\pi\colon S^3 \to S^2$ denote the Hopf fibration (see \eqref{hopffibration} below). 
The condition $[K,u]=0$ then implies that $u$ descends, i.e., there is a vector field $v$ on $S^2$ such that $v\circ \pi = D\pi \circ u$.
Since $u$ is divergence free, $v$ must also be divergence free, which implies  $v = \nabla^\perp \psi$ for some streamfunction $\psi$ on $S^2$.
The form~\eqref{axisymmetricexplicit} now follows by taking $\tilde\psi = \psi\circ\pi$ and observing that $[E_1, \tilde\sigma E_1] = 0$ if and only if $E_1\tilde\sigma = 0$.
}

We then find that the curl is given by 
$$ \omega = \curl{u} = (\tilde{\sigma}+(E_2^2+E_3^2)\tilde{\psi}) E_1 + \big( E_3\tilde{\sigma}\big) E_2 \textcolor{red}{-} \big(E_2\tilde{\sigma}\big) E_3.$$
%
%
The vorticity form~\eqref{3Dvorticityvelocity} of the Euler equation on $S^3$ then becomes the system
\begin{equation}\label{primitive3Deuler}
\partial_t(E_2^2+E_3^2)\tilde{\psi} +  \mathcal{B}\big( \tilde{\psi}, \tilde{\sigma}+(E_2^2+E_3^2)\tilde{\psi}\big) = 0, \qquad
\partial_t \tilde{\sigma} + \mathcal{B}\big(\tilde{\psi}, \tilde{\sigma}\big) = 0
\end{equation}
where 
\begin{equation}\label{Bpoisson}
\mathcal{B}(\tilde f,\tilde g) := (E_2\tilde f)(E_3\tilde g)-(E_3\tilde g)(E_2\tilde f)
\end{equation}
descends to the Poisson bracket on $S^2$\old{, as we shall see below}.
\new{Indeed, if $\tilde f = f\circ\pi$ and $\tilde g = g\circ\pi$ then direct calculations yield $\{f,g\} \circ \pi = \mathcal B(\tilde f,\tilde g)$.}

The map  $\Pi\colon\mathbb{R}^4\to\mathbb{R}^3$ given by 
\begin{equation}\label{hopffibration}
\Pi(w,x,y,z) = \new{\big( 2(wy+xz), 2(wz-xy), w^2+x^2-y^2-z^2\big)}
\end{equation}
takes $S^3$ into $S^2$ and its restriction $\pi\colon S^3\to S^2$ is the Hopf fibration. 
We compute that $D\pi(E_1)\equiv 0$, so the flow circles of $K=E_1$ all map to points in the quotient $S^3/S^1 \simeq S^2$. 
Thus, the conditions from \eqref{axisymmetricexplicit} that $\tilde{\sigma}$ and $\tilde{\psi}$ be $K$-invariant are precisely what one needs to have  real-valued functions $\sigma$ and $\psi$ defined on $S^2$ and satisfying $\sigma\circ\pi = \tilde{\sigma}$ and $\psi\circ\pi = \tilde{\psi}$. 
The equations \eqref{primitive3Deuler} then also descend to equations on $S^2$, given by 
\begin{equation}\label{3Deuleraxi}
\Delta \dot\psi + \{ \psi, \Delta \psi + \sigma\} = 0, \qquad \dot\sigma + \{\psi, \sigma\} = 0,
\end{equation}
where our choices on $S^3$ lead to exactly the standard Laplacian $\Delta$ and the standard Poisson bracket $\{\cdot,\cdot\}$ on $S^2$. 
Comparing to \eqref{2Dvorticitystream}, we see that the 3-D axisymmetric equation reduces to the 2-D equation when the swirl $\sigma$ is zero. 
\arxivonly{See Appendix \ref{descendmetricappendix} for details of these computations in spherical coordinates.}


\section{The product structure and its discretization}

If $u$ and $v$ are axisymmetric vector fields on $S^3$ then $[u,v]$ is again axisymmetric. 
Indeed, from the Jacobi identity
\begin{equation*}
   [[u,v],K] = -[\underbrace{[K,u]}_{0},v] - [\underbrace{[v,K]}_0,u] = 0.
\end{equation*}
Thus, the space of axisymmetric vector fields makes a Lie sub-algebra.
Here we construct the corresponding Lie algebra structure in terms of the components $(\psi,\sigma)\in C^\infty(S^2)\times C^\infty(S^2)$, for which the axisymmetric 3-D Euler equation \eqref{3Deuleraxi} on $S^2$ is the Euler-Arnold equation.
Once this structure is established, it becomes evident how to discretize it via Zeitlin's approach.

The Lie algebra is modelled on the product $T_\id\Diffmu(S^2)\times C^\infty(S^2)$, but with a more complicated Lie algebra than the usual product structures, i.e., the direct product, the semidirect product, or the central extension.
Instead, it is a special case of the \emph{Abelian extension}, described in detail by Vizman~\cite{Vi2008}.

\begin{definition}\label{abelianextdef}
   Let $\mathfrak{g}$ be a Lie algebra with a $\mathfrak{g}$-module $\Sigma$ specified by an action map $\rho\colon \mathfrak{g}\to \End(\Sigma)$.
   An Abelian extension of $\mathfrak{g}$ by $\Sigma$ is determined by a bilinear skew-symmetric map $b\colon \mathfrak{g}\times \mathfrak{g}\to \Sigma$ which satisfies the $2$-cocycle condition 
   \begin{equation}\label{cocycle}
   \sum_{\text{cyclic}} b([v_1,v_2], v_3) = \sum_{\text{cyclic}} \rho(v_1)b(v_2,v_3), \qquad v_1,v_2,v_3\in\mathfrak{g},
   \end{equation}
   where the cyclic sum is taken as in the Jacobi identity for the three vectors. 
   The Lie bracket on $\mathfrak{g}\times \Sigma$ is then defined by 
   \begin{equation}\label{bracketvizman}   
      [(v_1, \sigma_1),(v_2, \sigma_2)] = \big([v_1,v_2], \rho(v_1)\sigma_2 - \rho(v_2)\sigma_1 + b(v_1,v_2)\big).
   \end{equation}
\end{definition}

The bracket \eqref{bracketvizman} indeed gives a Lie algebra: antisymmetry is obvious, while the Jacobi identity follows from the usual Jacobi identity on $\mathfrak{g}$ and the cocycle condition on $b$ and $\rho$. 
Note that semidirect products correspond to $b=0$, while central extensions correspond to $\rho=0$.
In what follows, both $b$ and $\rho$ are nonzero.

\begin{proposition}\label{product3sphere}
With $\mathfrak{g} = T_{\id}\Diffmu(S^2)$ and $\Sigma=C^{\infty}(S^2,\mathbb{R})$, define the action 
\begin{equation*}
   \rho\colon\mathfrak{g}\to\End(\Sigma), \qquad   \rho(v)\sigma = \{\psi,\sigma\}\quad\text{for}\quad v = \sgrad \psi,
\end{equation*}
and the $2$-cocycle 
\begin{equation*}
   b\colon \mathfrak{g}\times \mathfrak{g}\to \Sigma, \qquad b(v_1,v_2) = -\{\psi_1,\psi_2\} \quad \text{for}\quad v_i=\sgrad \psi_i .
\end{equation*}
Then the Abelian extension in Definition~\ref{abelianextdef} reproduces the Lie algebra of axisymmetric volume preserving diffeomorphisms on $S^3$.
\end{proposition}

\begin{proof} 
By formula \eqref{axisymmetricexplicit}, we can write arbitrary elements $u_1, u_2$ in the Lie algebra of axisymmetric volume preserving diffeomorphisms of $S^3$ in the form 
$$ u_i = \tilde\sigma_i E_1 + \sgradbar\tilde \psi_i, 
\qquad \sgradbar f := -E_3(f) E_2 + E_2(f)E_3,$$
where $\tilde \psi_i$ and $\tilde \sigma_i$ are both $E_1$-invariant functions on $S^3$. 

From the bracket relations \eqref{S3brackets}, we get $[\tilde \sigma_1E_1,\tilde \sigma_2E_1]=0$, 
$$ [\sgradbar \tilde\psi, \tilde \sigma E_1] = \mathcal{B}(\tilde\psi,\tilde \sigma) E_1,$$
and 
$$ [\sgradbar \tilde\psi_1, \sgradbar \tilde\psi_2] = \sgradbar \mathcal{B}(\tilde\psi_1,\tilde\psi_2) - \mathcal{B}(\tilde\psi_1,\tilde\psi_2) E_1,$$
with $\mathcal{B}$ defined as in \eqref{Bpoisson}. 
Hence, we obtain 
$$ [u_1,u_2] 
= \big( \mathcal{B}(\tilde\psi_1,\tilde\sigma_2) + \mathcal{B}(\tilde\sigma_1,\tilde\psi_2) - \mathcal{B}(\tilde\psi_1,\tilde\psi_2) \big) E_1 + \sgradbar \mathcal{B}(\tilde\psi_1,\tilde\psi_2).$$
Identifying each $u_i$ with an ordered pair of functions $(\tilde\psi_i,\tilde\sigma_i)$, this formula tells us that 
$$ \big[(\tilde\psi_1,\tilde\sigma_1), (\tilde\psi_2,\tilde\sigma_2)\big] = \big(\mathcal{B}(\tilde\psi_1,\tilde\psi_2), \mathcal{B}(\tilde\psi_1,\tilde\sigma_2) + \mathcal{B}(\tilde\sigma_1,\tilde\psi_2) - \mathcal{B}(\tilde\psi_1,\tilde\psi_2)\big).$$
Under identifications via the Hopf projection $\pi\colon S^3\to S^2$ as in section~\ref{sub:axisymmetry_on_spheres},  
we have deduced the Lie algebra structure on $\mathfrak{g}\times 
\Sigma$
\begin{equation}\label{adjointproduct}
\big[ (\psi_1, \sigma_1), (\psi_2,\sigma_2)\big] = \big(\{\psi_1,\psi_2\}, \{\psi_1,\sigma_2\} + \{\sigma_1,\psi_2\} - \{\psi_1,\psi_2\} \big),
\end{equation}
and this is precisely the Lie algebra \eqref{bracketvizman}
with the given choices of $\rho$ and $b$.
\end{proof}

The $L^2$ kinetic energy metric on divergence-free velocity fields $u,v\in T_{\id}\Diffmu(M)$ of a Riemannian manifold $(M,g)$ is given by 
$$ \langle u,v\rangle = \int_M g(u,v)\,\mu.$$
For axisymmetric divergence-free fields $u$ on $S^3$ represented by \eqref{axisymmetricexplicit} this yields
\begin{equation}\label{L2kineticS3}
 \langle u,u\rangle = \int_{S^3} \tilde{\sigma}^2 
+ (E_2\tilde{\psi})^2 + (E_3\tilde{\psi})^2 \, .
\end{equation}
We can compute that for $E_1$-invariant functions $\tilde{\psi}$, 
$$ E_2^2(\tilde{\psi}) + E_3^2(\tilde{\psi}) = \Delta \psi,$$
in terms of the usual Laplacian on $S^2$,
and thus the metric \eqref{L2kineticS3} reduces to 
\begin{equation}\label{metricdiffeoproduct}
\langle u, u\rangle = 4\pi \int_{S^2} 
\sigma^2 + \lvert \nabla \psi\rvert^2 \, .
\end{equation}
\arxivonly{See Appendix \ref{descendmetricappendix} for details of these computations.}

The following proposition is essentially the statement that axisymmetric volume preserving diffeomorphisms constitute a totally geodesic subgroup of all volume preserving diffeomorphisms.
We provide an explicit derivation, since it also applies to the corresponding Zeitlin product.

\begin{proposition}\label{eulerarnolddiffeoproduct}
   The Euler--Arnold equation for the Lie algebra $T_\id\Diffmu(S^2)\times C^\infty(S^2)$, with bracket \eqref{adjointproduct} and right-invariant metric \eqref{metricdiffeoproduct}, is given by the equations \eqref{3Deuleraxi}.
   They describe axisymmetric solutions to the Euler equations on $S^3$.
\end{proposition}

\begin{proof}
Using \eqref{adjointproduct} and the fact that the adjoint operator is the negative of the Lie bracket of the right-invariant vector fields, we have for any functions $g$ and $\sigma$, and any mean-zero functions $f$ and $\psi$, that 
\begin{equation}\label{adjointnegsign}
\ad_{(\psi,\sigma)}(f,g) = \big(  -\{\psi,f\}, -\{\psi,g\} -\{\sigma,f\} + \{\psi,f\} \big).
\end{equation}
The Euler--Arnold equation is given by 
$$ \partial_t(\psi,\sigma) + \adstar_{(\psi,\sigma)}(\psi,\sigma) = 0. $$
Consequently, $(\psi,\sigma)$ satisfies the Euler--Arnold equation if and only if for every pair of functions $(f,g)$ we have 
\begin{equation}\label{weakeulerarnold}
EA:=\langle \partial_t(\psi,\sigma), (f,g)\rangle 
+ \langle (\psi,\sigma), \ad_{(\psi,\sigma)}(f,g) \rangle = 0.
\end{equation}
From \eqref{metricdiffeoproduct} we then obtain
$$
EA
= \int_{S^2}  g\dot\sigma - f\Delta\dot\psi 
+ \sigma (-\{\psi,g\} -\{\sigma,f\} + \{\psi,f\}) + \Delta\psi \{ \psi,f\}  \, .
$$
Now using the formula $\int_{S^2} (f\{g,h\}+h\{g,f\}) = 0$, which is essentially an integration by parts using Stokes' Theorem, we obtain 
$$ EA 
= \int_{S^2} g\big(\dot\sigma + \{\psi,\sigma\}\big)
- f\big( \Delta\dot\psi + \{\psi, \sigma\} + \{\psi, \Delta \psi\}\big) \, ,$$
and this is zero for every $f$ and $g$ if and only if $\psi$ and $\sigma$ satisfy the equations \eqref{3Deuleraxi}.
\end{proof}

\subsection{Casimir functions}

In addition to the Hamiltonian, Euler--Arnold equations conserve the Casimir functions associated with the Lie--Poisson structure.
For the axisymmetric Euler equation~\eqref{3Deuleraxi} there are infinitely many Casimir functions, corresponding to the magnetic swirls and cross-helicity in 2-D incompressible magneto-hydrodynamics~\cite{VD,MoGr1980}.
Thus, the situation for axisymmetric Euler equations is quite different from the full 3-D case, where there are only finitely many independent Casimirs.

\begin{proposition}\label{prop:casimirs}
   Consider the Lie algebra $\mathfrak{g}\times\Sigma$ in Proposition~\ref{product3sphere}.
   For an arbitrary $f\in C^\infty(\mathbb{R})$, the functionals on $(\mathfrak{g}\times\Sigma)^\star\simeq \mathfrak{g}\times\Sigma$ given by
   \begin{equation*}
      C_f = \int_{S^2} f\circ\sigma , \qquad \old{I = \int_{S^2 } (\Delta\psi)\sigma} \new{I_f = \int_{S^2 } (\Delta\psi)f\circ\sigma}  ,
   \end{equation*}
   are Casimir functions for the corresponding Lie--Poisson structure on $(\mathfrak{g}\times\Sigma)^\star$.
\end{proposition}

\begin{proof}
   From the governing equations~\eqref{3Deuleraxi} we obtain, first 
   for $C_f$ that
   \begin{equation*}
      \frac{d}{dt}C_f = \langle f'\circ\sigma, \dot\sigma\rangle_{L^2} = \langle f'\circ\sigma, -\{\psi,\sigma \}\rangle_{L^2} = 
      \langle \underbrace{\{f'\circ\sigma, \sigma\}}_{0},\psi \rangle_{L^2} = 0,
   \end{equation*}
   and then for \old{$I$}\new{$I_f$} that
   \old{
   \begin{align*}
      \frac{d}{dt}I &= \langle \Delta\psi, \dot\sigma\rangle_{L^2} + \langle \Delta\dot\psi,\sigma\rangle_{L^2} = 
      \langle \Delta\psi, \dot\sigma\rangle_{L^2} - \langle \{ \psi, \sigma + \Delta \psi\}, \sigma\rangle_{L^2} = \\
      & -\langle \Delta\psi, \{\psi,\sigma \}\rangle_{L^2} - \langle \{ \psi, \Delta \psi\}, \sigma\rangle_{L^2} = 0.
   \end{align*}
   }
   \new{
   \begin{align*}
      \frac{d}{dt}I_f &= \langle \Delta\psi, \frac{d}{dt}f\circ\sigma)\rangle_{L^2} + \langle \Delta\dot\psi,f\circ\sigma\rangle_{L^2} = \\
      & \langle \Delta\psi, \frac{d}{dt}f\circ\sigma\rangle_{L^2} - \langle \{ \psi, \sigma + \Delta \psi\}, f\circ\sigma\rangle_{L^2} = \\
      & -\langle \Delta\psi, \{\psi,f\circ\sigma \}\rangle_{L^2} - \langle \{ \psi, \Delta \psi\}, f\circ\sigma\rangle_{L^2} = 0.
   \end{align*}
   }
\end{proof}

\subsection{Spatial discretization via Zeitlin's approach}

We now turn to our main point: a Zeitlin  discretization for the Euler--Arnold structure in Proposition~\ref{eulerarnolddiffeoproduct}.


\begin{theorem}\label{productzeitlinthm}
Let $\mathfrak{g}=\mathfrak{su}(n)$ equipped with the scaled commutator bracket $\frac{1}{\hbarN}[\cdot,\cdot]$.
With $\Sigma = \mathfrak{u}(n)$, define the action $\rho$ of $\mathfrak{g}$ on $\Sigma$ by $\rho(P)\sigmamat = \frac{1}{\hbarN}[P,\sigmamat]$, and define a $2$-cocycle $b\colon \mathfrak{g}\times \mathfrak{g}\to \Sigma$ by $b(P_1,P_2) = -\frac{1}{\hbarN}[P_1,P_2]$.
Consider then the Abelian extension in Definition~\ref{abelianextdef} with the Lie bracket~\eqref{bracketvizman}.
Define an inner product on $\mathfrak{su}(n)\times \mathfrak{u}(n)$ by
\begin{equation}\label{zeitlinprodmet}
\big\langle (P_1,\sigmamat_1), (P_2,\sigmamat_2)\big\rangle =   \operatorname{tr}(P_1 \Delta_n P_2) - \operatorname{tr}(\sigmamat_1 \sigmamat_2),
\end{equation}
where $\Delta_n$ is the Hoppe-Yau Laplacian \eqref{hoppeyau}.
Then the corresponding Euler-Arnold equation is 
\begin{equation}\label{zeitlinproduct}
\Delta_n \dot P + \frac{1}{\hbarN}[P, \Delta_n P + \sigmamat]
=0, \qquad
\dot \sigmamat + \frac{1}{\hbarN}[P, \sigmamat] = 0.
\end{equation}
\end{theorem}

\begin{proof}
   The $\ad$ operator is the negative of the Lie bracket
   \begin{equation*}
      \ad_{(P,\sigmamat)}(U,V) = \Big( -\frac{1}{\hbarN}[P,U], -\frac{1}{\hbarN}[P,V] + \frac{1}{\hbarN}[U,\sigmamat] + \frac{1}{\hbarN}[P,U] \Big).
   \end{equation*}
   We compute the analogue of \eqref{weakeulerarnold} by the same method as in the proof of Proposition \ref{eulerarnolddiffeoproduct}, using bi-invariance of the trace metric: 
   \begin{align*} 
   EA_n :&= \langle ( \dot P, \dot \sigmamat), (U, V)\rangle 
   + \langle ( P, \sigmamat), \ad_{(P,\sigmamat)}(U,V)\rangle \\
   &= \operatorname{tr}\big(\dot P\Delta_n U \big) - \operatorname{tr}\big(\dot \sigmamat V\big) - \operatorname{tr}\big( (\Delta_n P)\frac{1}{\hbarN}[P, U] \big) + \operatorname{tr}\big(\sigmamat (\frac{1}{\hbarN}[P,V] - \frac{1}{\hbarN}[U,\sigmamat] - \frac{1}{\hbarN}[P,U])\big) \\
   &= \operatorname{tr}\big(\Delta_n\dot P U \big) - \operatorname{tr}\big(\dot \sigmamat V\big) + \operatorname{tr}\big( \frac{1}{\hbarN}[P,\Delta_n P]U \big) - \operatorname{tr}\big(\frac{1}{\hbarN}[P,\sigmamat]V\big) + \operatorname{tr}\big(\frac{1}{\hbarN}[P,\sigmamat]U\big) \\
   &= \operatorname{tr}\big((\Delta_n\dot P + \frac{1}{\hbarN}[P,\Delta_n P] + \frac{1}{\hbarN}[P,\sigmamat]) U \big) - \operatorname{tr}\big((\dot \sigmamat+\frac{1}{\hbarN}[P,\sigmamat]) V\big) .
   \end{align*}
   This is zero for all $(V,U)\in \mathfrak{u}(N)\times \mathfrak{su}(n)$ if and only if equations \eqref{zeitlinproduct} are satisfied.
\end{proof}

From quantization theory we know that if $P_1,P_2\in \mathfrak{su}(n)$ are related to $\psi_1,\psi_2\in C^\infty(S^2)$ via the quantization $\mathcal T_n$ described in Section~\ref{zeitlinsection}, then $\frac{1}{\hbarN}[P_1, P_2] \to \mathcal T_n \{ \psi_1,\psi_2\}$ as $n\to\infty$ in the spectral norm on $\mathfrak{su}(n)$ (see \cite{ChPo2018} for details).
Thus, the equations \eqref{zeitlinproduct} provide a spatial discretization of the $S^3$ axisymmetric Euler equations~\eqref{3Deuleraxi}.

Due to the Euler--Arnold structure, the discretized equations~\eqref{zeitlinproduct} preserve analogues of the Casimir functions in Proposition~\ref{prop:casimirs}.

\begin{proposition}\label{prop:zeitlin_casimirs}
   With $\mathfrak{g}\times\Sigma$ as in Theorem~\ref{productzeitlinthm}, the Casimir functions are
   \begin{equation*}
      C^n_f = \operatorname{tr}(f(i\sigmamat)), \qquad I_f^n = i \operatorname{tr}(f(i\sigmamat)\Delta_n P)
   \end{equation*}
   where $f$ is an arbitrary real analytic function.
   These functions are thus conserved by the Euler--Arnold equations~\eqref{zeitlinproduct} on $\mathfrak{g}\times\Sigma$.
\end{proposition}

\section{The Jacobi equation}

Now we consider some geometric aspects of the 3-D Zeitlin model \eqref{zeitlinproduct}. Recall that the 
Jacobi equation along geodesics is given by equation \eqref{jacobisplit}.
It describes stable perturbations, which lead to conjugate points, but also possible instabilities. 
We can linearize the equations \eqref{zeitlinproduct} for perturbations $\sigmamat+\epsilon Z_1$ and $P + \epsilon Z_2$ to obtain
\begin{equation}\label{lineulerzeitlin}
\begin{split}
\dot Z_1(t) + \frac{1}{\hbarN}[ P(t), Z_1(t)] + \frac{1}{\hbarN}[Z_2(t), \sigmamat(t)] &= 0 \\ 
\qquad \Delta_n \dot Z_2(t) + \frac{1}{\hbarN}[Z_2(t), \Delta_n P(t)]
+ \frac{1}{\hbarN}[ P(t), \Delta_n Z_2(t)]
+ \frac{1}{\hbarN}[ P(t), Z_1(t)]
+ \frac{1}{\hbarN}[Z_2(t), \sigmamat(t)]&=0.
\end{split}
\end{equation}
Similarly, using the formula \eqref{adjointproduct} for the Lie bracket, the linearized flow equation \eqref{jacobisplit} for a 
Jacobi field $J$ with right translated generators $Y_1$ and $Y_2$ takes the form
\begin{equation}\label{linflowzeitlin}
\begin{split}
\dot Y_1(t) + \frac{1}{\hbarN}\big( [\sigmamat(t), Y_2(t)]+[P(t),Y_1(t)] - [P(t),Y_2(t)]\big) &= Z_1(t) \\
\dot Y_2(t) + \frac{1}{\hbarN}[P(t),Y_2(t)] &= Z_2(t).
\end{split}
\end{equation}
Our goal in this section is to illustrate how to solve this system of equations in a simple case. 

Steady solutions of the Euler-Arnold equation \eqref{zeitlinproduct} are given by matrices $(\sigmamat,P)$ satisfying 
$$ [ P, \sigmamat] = 0, \qquad [ P, \Delta_n P] = 0.$$
A simple way to satisfy these equations is to take $P=\hbarN S_3$ as in Section \ref{zeitlinsection}, since in that 
case we have $\Delta_n P = -2P$, and we also take $\sigmamat=\hbarN S_3$. This corresponds to taking $\sigma=\psi=-\cos{\theta}$
on the $2$-sphere in the equations \eqref{3Deuleraxi}, so that the underlying 2-D flow on the 2-D sphere is the rigid 
rotation by $\nabla^{\perp}\psi = \partial_{\phi}$ in the usual spherical coordinates $(\theta,\phi)$. The reason taking $\sigmamat=P$
is the simplest choice is that it reduces the first equation in \eqref{linflowzeitlin} to the same form as the second, as we will see.


\begin{theorem}\label{zeitlinjacobithm}
For the Euler velocity field in $\mathfrak{su}(n)\times\mathfrak{u}(n)$ given by $P(t)=\hbarN S_3$ and $\sigmamat(t)=\hbarN S_3$, let $\gamma(t)$ with $\gamma(0)=e$ be the corresponding geodesic curve in the Lie group.
For each positive integers $m$, $\ell$, $k$ with $m\le \ell\le (n-1)/2$, there are conjugate points \old{$\gamma(T)$}\new{$\gamma(t)$} to the identity at times $t=\frac{4\pi k\ell}{m}$ and $t=\frac{4k\pi(\ell+1)}{m}$. Each of these occurs with multiplicity $2$ for each distinct pair $(\ell,m)$ of positive integers. 
\end{theorem}

\begin{proof}
Using $P(t)=\hbarN S_3$ and $\sigmamat(t)=\hbarN S_3$, and writing 
$$ \adc:= \ad_{S_3},$$ 
the linearized Euler equation \eqref{lineulerzeitlin} becomes 
\begin{equation}\label{linearizedeulersimplest}
   \dot Z_1 = \adc(Z_2-Z_1), \qquad 
   \Delta_n \dot Z_2 = -\adc (Z_1+Z_2+\Delta_nZ_2). 
\end{equation}
Meanwhile, the linearized flow equation \eqref{linflowzeitlin} is given by
\begin{equation}\label{linearizedflowsimplest}
   \dot Y_1 + \adc Y_1 = Z_1, \qquad \dot Y_2+\adc Y_2 = Z_2.
\end{equation}
To solve equations \eqref{linearizedeulersimplest}--\eqref{linearizedflowsimplest}, it is convenient to define an operator 
\begin{equation}\label{DdefJacobi}
\dbar := \sqrt{-\Delta_n+\tfrac{1}{4}I}-\tfrac{1}{2}I, \qquad \dbar(T_{\ell,m}) = \ell T_{\ell,m} \quad \forall\; 0\le \ell\le s, \lvert m\rvert \le \ell.
\end{equation}
Note that $\Delta_n=-\dbar(I+\dbar)$. 
Since $\Delta_n$ commutes with $\adc$, so does $\dbar$.
 
We define the new variables 
\begin{equation}\label{Znewdef}
Z_3 = Z_1-\dbar Z_2, \qquad Z_4 = Z_1+(\dbar+I)Z_2, \qquad Y_3 = Y_1-\dbar Y_2, \qquad Y_4 = Y_2+(\dbar+I)Y_1
\end{equation}
and observe that the equations \eqref{linearizedeulersimplest} can be rewritten in the form
\begin{align*}
(\dbar+I)\,\tfrac{d}{dt}(Z_1-\dbar Z_2) &= 
(\dbar+I)(\adc(Z_2-Z_1) - \adc(Z_1 + Z_2 - (\dbar+1)\dbar Z_2) \\
&= -(\dbar+2I)\adc(Z_1-\dbar Z_2) 
\end{align*}
which implies that 
\begin{equation}\label{Z3eq}
(\dbar+I)\tfrac{d}{dt} Z_3 = -(\dbar+2I)\adc Z_3.
\end{equation}

Similarly, we obtain 
\begin{equation}\label{Z4eq}
\dbar\tfrac{d}{dt}Z_4 = -(\dbar-I)\adc Z_4.
\end{equation}
We also see that \eqref{linearizedflowsimplest} takes the form 
\begin{equation}\label{Y34eq}
\tfrac{d}{dt} Y_3 + \adc Y_3 = Z_3, \qquad \tfrac{d}{dt} Y_4 + \adc Y_4 = Z_4.
\end{equation}

We conclude from \eqref{Z3eq} that if $Z_3(0)=0$ then $Z_3(t)=0$ for all $t\ge 0$, and thus by \eqref{Y34eq} that $Y_3(t)=0$ for all $t$ since $Y_3(0)=0$. 
Similarly, if $Z_4(0)=0$, then $Y_4(t)=0$ for all $t$. 
Furthermore, since both $\adc$ and $\dbar$ are block-diagonal in the basis $T_{\ell, m}$, with 
$$ \adc T_{\ell, m} = m   T_{\ell,-m}, \qquad \dbar T_{\ell, m} = \ell T_{\ell, m}, \qquad -\ell \le m\le \ell,$$
we can write equation \eqref{Z3eq} in block diagonal form. That is, writing 
$$ Z_3(t) = \sum_{\ell=0}^s \sum_{m=-\ell}^{\ell} a_{\ell, m}(t) T_{\ell, m}, \qquad Y_3(t) = \sum_{\ell=0}^s \sum_{m=-\ell}^{\ell} c_{\ell, m}(t) T_{\ell, m},$$
we obtain the system 
$$ a_{\ell,m}'(t) = \frac{(\ell+2)m}{\ell+1}\, a_{\ell,-m}(t), \qquad c_{\ell,m}'(t) - m c_{\ell,-m}(t) = a_{\ell,m}(t), \qquad -\ell\le m\le \ell.$$

If $m\ne 0$, the solutions with $c_{\ell,m}(0)=0$ are easily found to be 
$$ c_{\ell,m}(t) = \frac{2(\ell+1)}{m}\, \sin{\left( \frac{mt}{2(\ell+1)}\right)} \left[ a_{\ell,m}(0) \cos{\left(\frac{(2\ell+3)m t}{2(\ell+1)}\right)} + a_{\ell,-m}(0) \sin{\left(\frac{(2\ell+3)m t}{2(\ell+1)}\right)} \right].$$
Hence we get conjugate points occurring at times $t=\frac{4k \pi(\ell+1)}{m}$, with multiplicity two in each block. Obviously if $m=0$ we simply get $c_{\ell,m}(t) =  a_{\ell,m}(0)t$, and there are no conjugate points arising from such initial conditions. 

Similarly solving the system \eqref{Z4eq}--\eqref{Y34eq} for $Z_4(t) = \sum b_{\ell,m}T_{\ell,m}$ and $Y_4(t) = \sum d_{\ell,m}(t) T_{\ell,m}$ gives 
$$ d_{\ell,m}(t) = \frac{2\ell}{m} \, \sin{\left(\frac{mt}{2\ell}\right)} \left[ 
b_{\ell,m}(0) \cos{\left(\frac{(2\ell-1)m t}{2\ell}\right)} + b_{\ell,-m}(0) \sin{\left(\frac{(2\ell-1)m t}{2\ell}\right)} \right],$$
and we obtain conjugate points at $t=\frac{4k\pi\ell}{m}$ for every positive integer $k$, in each block.
\end{proof}
 
The reason the analysis is particularly simple in this case is that the corresponding vector field on the $3$-sphere is a Killing field, and the combinations $Z_1+(\dbar+I)Z_2$ and $Z_1-\dbar Z_2$ occur naturally when one is computing curl eigenfields. 
See \cite{firstconjugate} for details, where the conjugate points are worked out explicitly along a similar geodesic (however we note that in that paper one considers the full volume preserving diffeomorphism group, not the axisymmetric subgroup, so there are fewer conjugate points in the present case).

\section{Numerical experiments}\label{sec:numerics}

Here we give two numerical experiments for the 3-D axisymmetric Zeitlin model \eqref{zeitlinproduct}.\footnote{A Python-based code for the simulations is available at \href{https://github.com/klasmodin/quflow}{github.com/klasmodin/quflow}. 
}

To retain the structural benefits of the Zeitlin based spatial discretization, it is essential to use a temporal discretization that preserves the underlying Lie--Poisson structure, which in turn implies conservation of Casimir functions.
Since the Casimirs for the $S^3$ axisymmetric Euler equations \eqref{3Deuleraxi} coincide with those for the 2-D incompressible magnetohydrodynamic (MHD) equations, we use the Casimir-preserving numerical integration scheme for the Zeitlin discretization of MHD, developed by Modin and Roop~\cite{MoRo2024}.
Thereby, the benefits of the spatial discretization remain in the fully discretized system of equations.

In addition to visualizations of the fields $\Delta\psi$ and $\sigma$, we demonstrate the growth of the supremum norm of the vorticity vector
\begin{equation*}
   \lVert \omega \rVert_{\infty} = \sup_{\tilde x\in S^3} \lvert\omega(\tilde x) \rvert = 
   \sup_{x\in S^2} \sqrt{(\Delta \psi + \sigma)^2 + \lvert \nabla\sigma\rvert^2}.
\end{equation*}
The analogous formula for Zeitlin's model is
\begin{equation}\label{eq:sup-norm-zeitlin}
   \lVert (\Delta_n P, \sigmamat)\rVert_{\infty} = \sqrt{\lVert -(\Delta_n P + \sigmamat)^2-\sum_{\alpha=1}^3 (\nabla_n^\alpha \sigmamat)^2\rVert},
\end{equation}
where $\lVert\cdot\rVert$ denotes the spectral norm and $\nabla_n^\alpha \sigmamat = [S_\alpha,\sigmamat]$ for $S_\alpha$ as in Section~\ref{zeitlinsection}.

See \cite{CiViMo2023} for details on how to efficiently compute $\mathcal T_n\psi$, the corresponding pseudo-inverse $\mathcal T_n^{-1}P$ (to obtain visualizations), and solution to the quantized Poisson equation $\Delta_n P = W$.

\subsection{First simulation: smooth, symmetric data}

Let $Y_{\ell,m}\in C^\infty(S^2)$ denote the real spherical harmonics.
The initial data are
\begin{equation*}
   \Delta\psi\big|_{t=0} = Y_{2,1}, \qquad \sigma\big|_{t=0} = Y_{1,0}.
\end{equation*}
These data are antisymmetric under reflection in the equatorial plane.
Consequently, the geometry corresponds to a hemisphere with no-slip boundary conditions along the equator.

Visualizations of $\Delta_n P$ and $\sigmamat$ at various output times are given in Figure~\ref{fig:blowup-evolution} for $n=1024$.
We see the formation of a shock wave in $\Delta_n P$, growing in magnitude, and a corresponding sharp gradient front in $\sigmamat$.
This formation indicates fast growth of the sup-norm \eqref{eq:sup-norm-zeitlin}.
Indeed, in Figure~\ref{fig:blowup-Linf-norm} the growth is slightly faster than exponential until the resolution allowed by $n$ is unable to resolve the increasingly steep shock wave front.

\begin{figure}
   \centering
   \begin{subfigure}{0.49\textwidth}
      \caption*{$\Delta_n P$ at $t=0$} 
      \includegraphics[width=\textwidth]{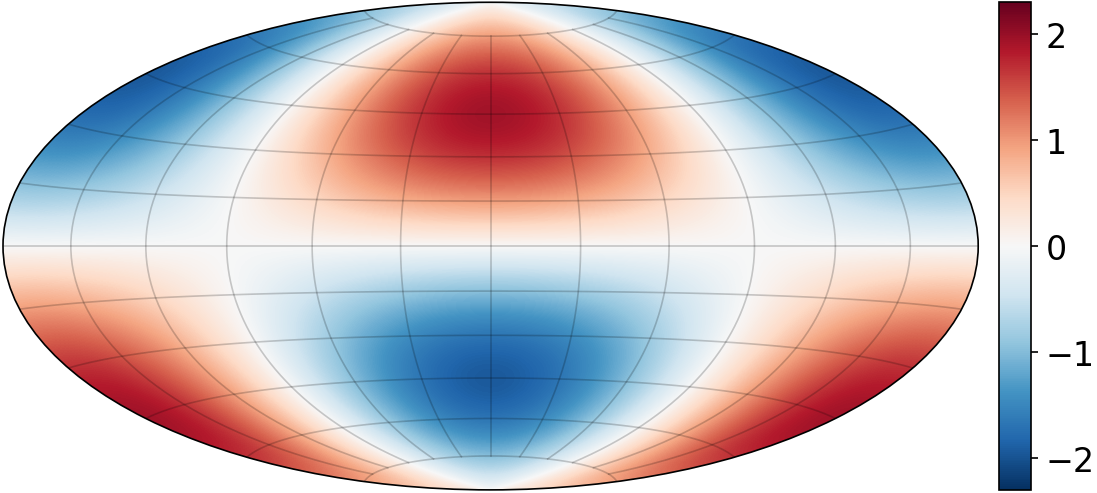}
   \end{subfigure}
   \begin{subfigure}{0.49\textwidth}
      \caption*{$\sigmamat$ at $t=0$} 
      \includegraphics[width=\textwidth]{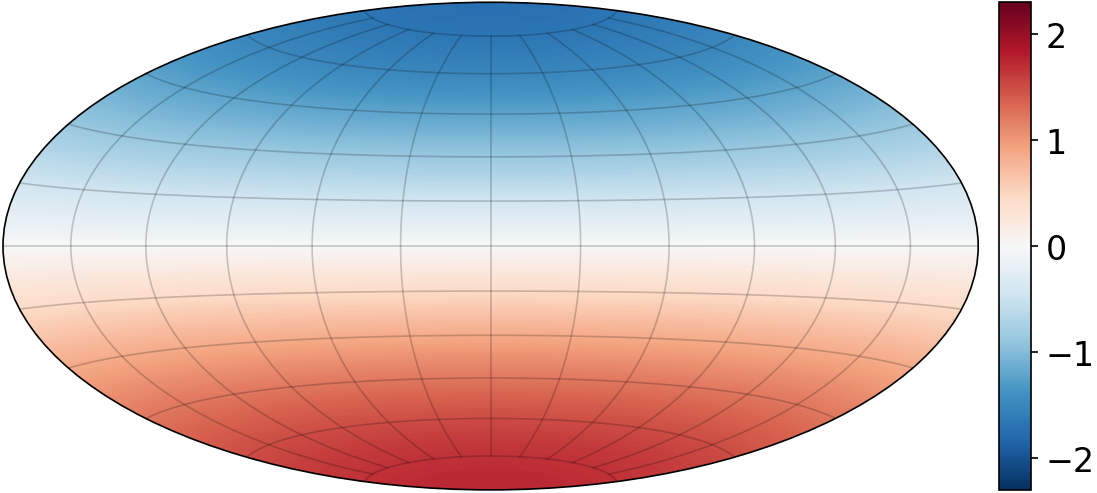}
   \end{subfigure}
   \\
   \begin{subfigure}{0.49\textwidth}
      \caption*{$\Delta_n P$ at $t=10$} 
      \includegraphics[width=\textwidth]{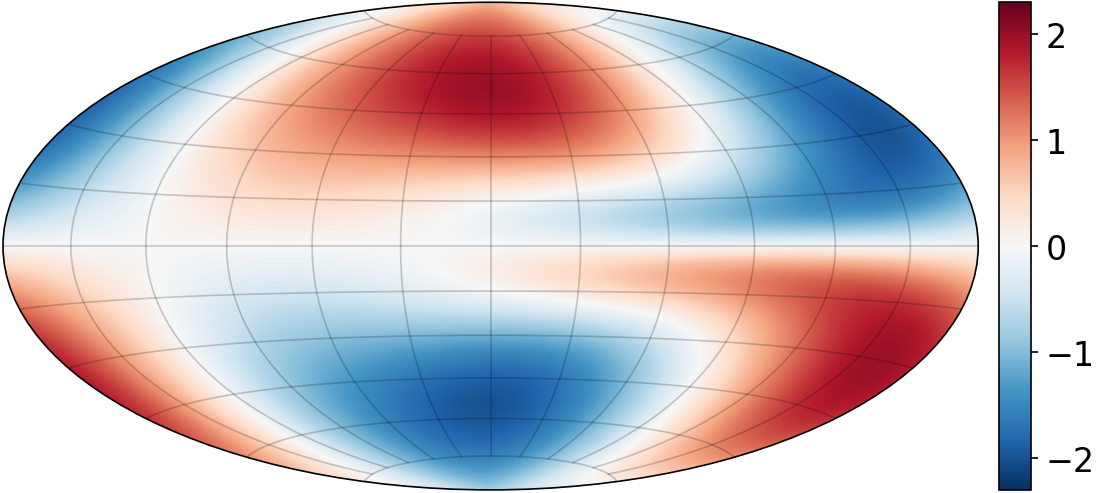}
   \end{subfigure}
   \begin{subfigure}{0.49\textwidth}
      \caption*{$\sigmamat$ at $t=10$} 
      \includegraphics[width=\textwidth]{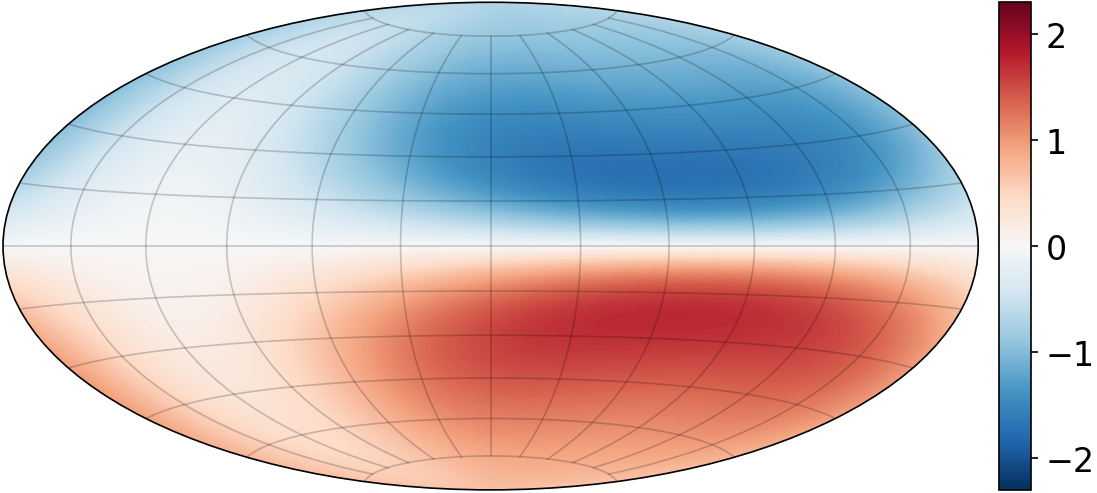}
   \end{subfigure}
   \\
   \begin{subfigure}{0.49\textwidth}
      \caption*{$\Delta_n P$ at $t=20$} 
      \includegraphics[width=\textwidth]{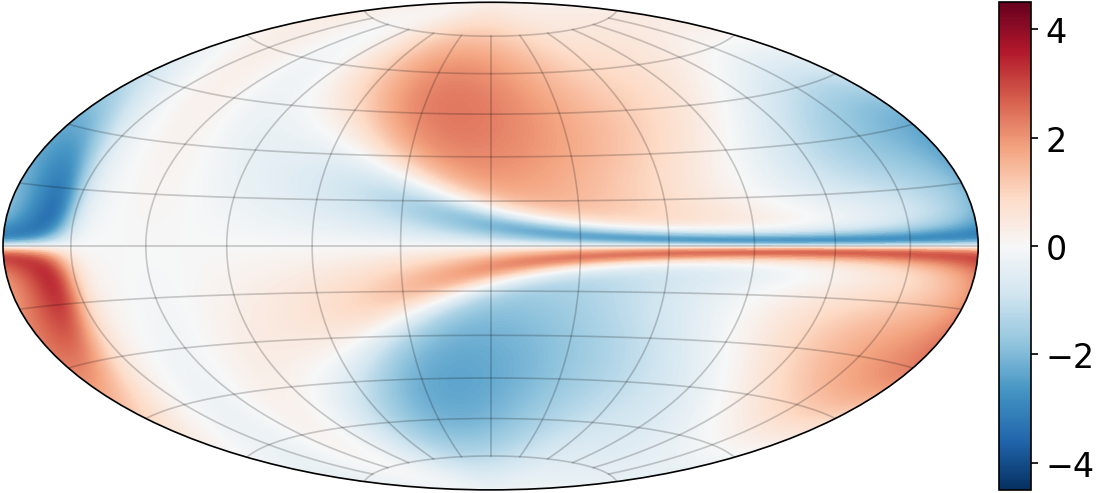}
   \end{subfigure}
   \begin{subfigure}{0.49\textwidth}
      \caption*{$\sigmamat$ at $t=20$} 
      \includegraphics[width=\textwidth]{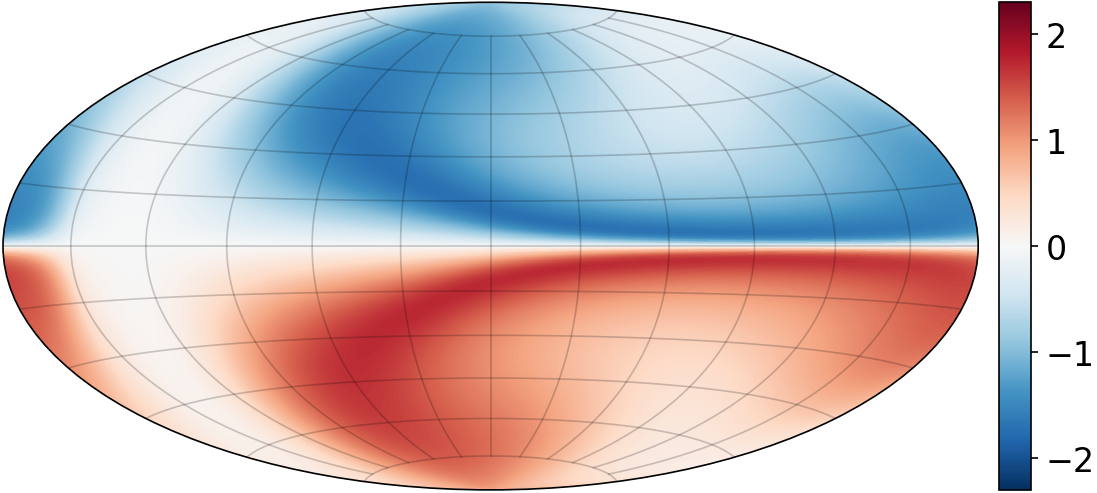}
   \end{subfigure}
   \\
   \hspace{0.1em}
   \begin{subfigure}{0.49\textwidth}
      \caption*{$\Delta_n P$ at $t=30$} 
      \includegraphics[width=\textwidth]{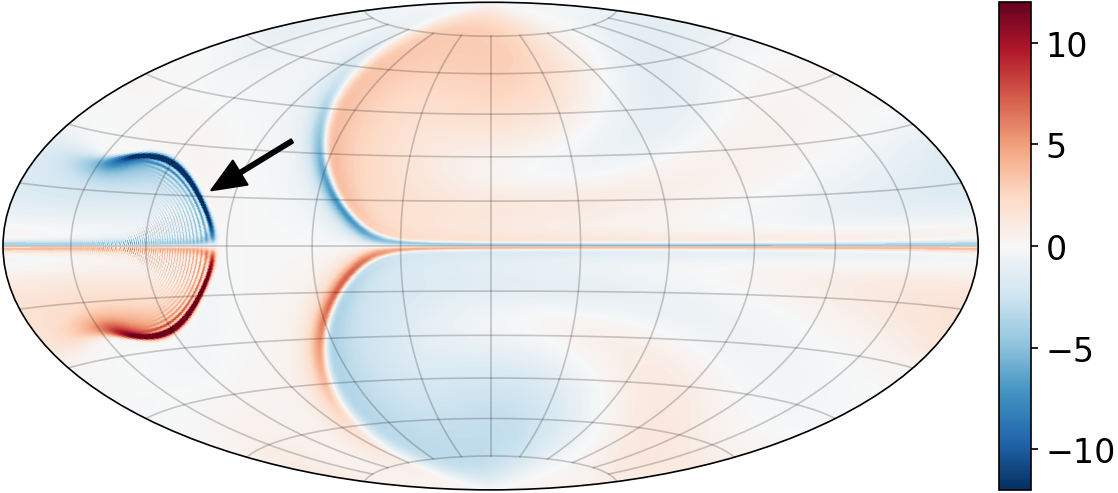}
   \end{subfigure}
   \hspace{-0.6em}
   \begin{subfigure}{0.49\textwidth}
      \caption*{$\sigmamat$ at $t=30$} 
      \includegraphics[width=\textwidth]{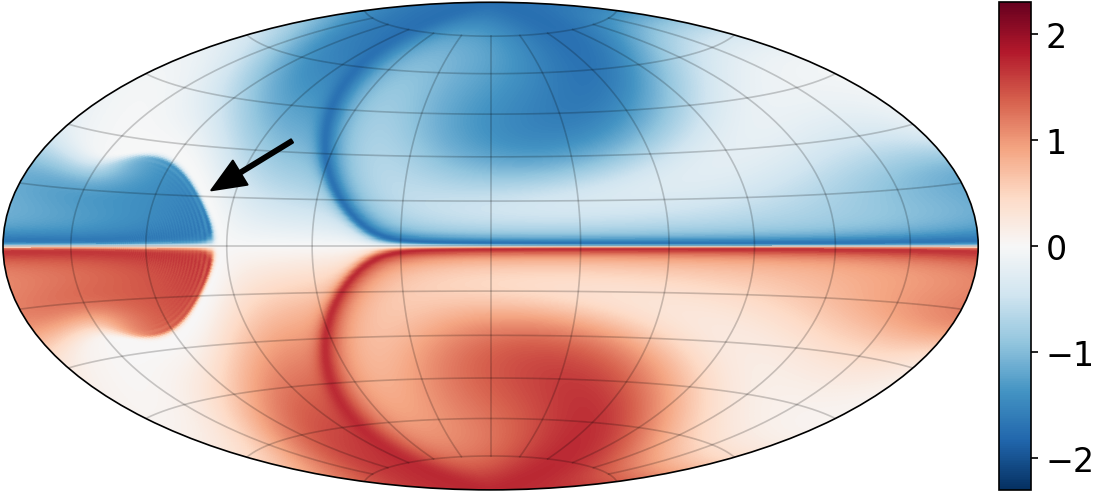}
   \end{subfigure}
   \caption{
      First simulation with $n=1024$.
      Visualization of the evolution of $\Delta_nP(t)$ and $\sigmamat(t)$.
      Notice the formation of a shock-wave in $\Delta_nP$ and a corresponding sharp gradient of $\sigmamat$, which implies rapid growth of the supremum norm of the vorticity vector $\omega$.
   }
   \label{fig:blowup-evolution}
\end{figure}

\begin{figure}
   \centering
   \begin{subfigure}{0.8\textwidth}
      \caption*{Evolution of the vorticity supremum norm for different $n$}\label{subfig:evolution-Linf} 
      \includegraphics[width=\textwidth]{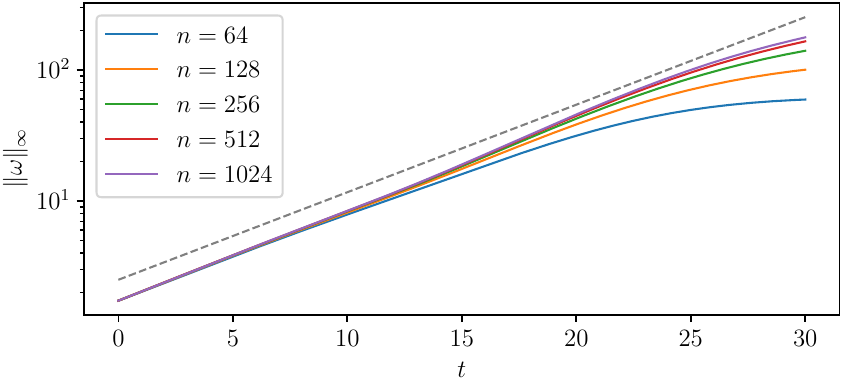}
   \end{subfigure}
   \caption{
      First simulation. The supremum norm of the 3-D vorticity $\omega$ for different choices of $n$.
      While the exact dynamics is accurately resolved, the growth is somewhat faster than exponential (the dotted line).
      However, eventually the curve flattens out when the model at level $n$ ceases to resolve the sharpness of the shock wave seen in Figure~\ref{fig:blowup-evolution}.
      Indeed, this flattening eventually occurs for any $n$, as all norms are equivalent in finite dimension and the energy norm is bounded.
      The time when this flattening begins is thus an indication that the sharpness of the shock wave is no longer accurately resolved.
   }
   \label{fig:blowup-Linf-norm}  
\end{figure}

\subsection{Second simulation: smooth, random data}

Here, the initial data are of the form
\begin{equation*}
   \Delta\psi\big|_{t=0} = \sum_{\ell=0}^{10} \sum_{m=-\ell}^\ell a_{\ell,m} Y_{\ell,m}, \qquad \sigma\big|_{t=0} = \sum_{\ell=0}^{10} \sum_{m=-\ell}^\ell b_{\ell,m} Y_{\ell,m},
\end{equation*}
where the coefficients $a_{\ell,m}$ and $b_{\ell,m}$ are drawn as independent samples from the standard Gaussian distribution.
This setup represents generic, smooth initial configurations.

Visualizations of $\Delta_n P$ and $\sigmamat$ at various output times are given in Figure~\ref{fig:random-evolution} for $n=1024$.
For 2-D Euler on $S^2$, generic initial conditions give rise interacting coherent blob structures~\cite{MoVi2020,MoVi2024}.
For the axisymmetric 3-D Euler on $S^3$ the situation is different.
Indeed, all the large scale structure of $\Delta_n P$ and $B$ disperse into higher frequency components, as captured in Figure~\ref{fig:random-evolution} at $t=20$.
Eventually, depending on $n$, the dispersion cannot continue further, due to the finite dimensionality of the model, so the sup-norm of $\omega$ flattens out, as seen in Figure~\ref{fig:random-Linf-norm}.
Initially it grows exponentially or faster.

\begin{figure}
   \centering
   \begin{subfigure}{0.49\textwidth}
      \caption*{$\Delta_n P$ at $t=0$} 
      \includegraphics[width=\textwidth]{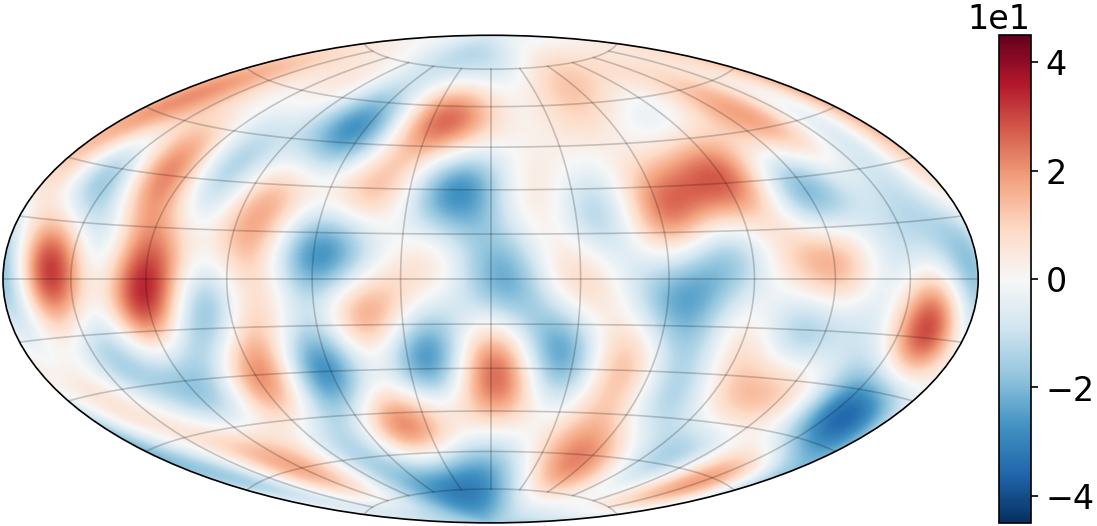}
   \end{subfigure}
   \begin{subfigure}{0.49\textwidth}
      \caption*{$\sigmamat$ at $t=0$} 
      \includegraphics[width=\textwidth]{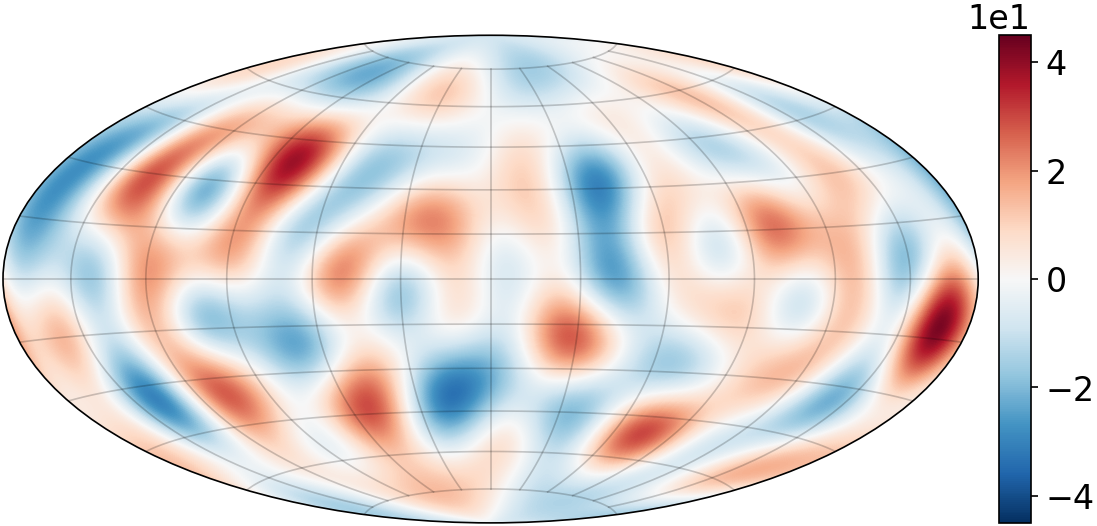}
   \end{subfigure}
   \\
   \begin{subfigure}{0.49\textwidth}
      \caption*{$\Delta_n P$ at $t=1$} 
      \includegraphics[width=\textwidth]{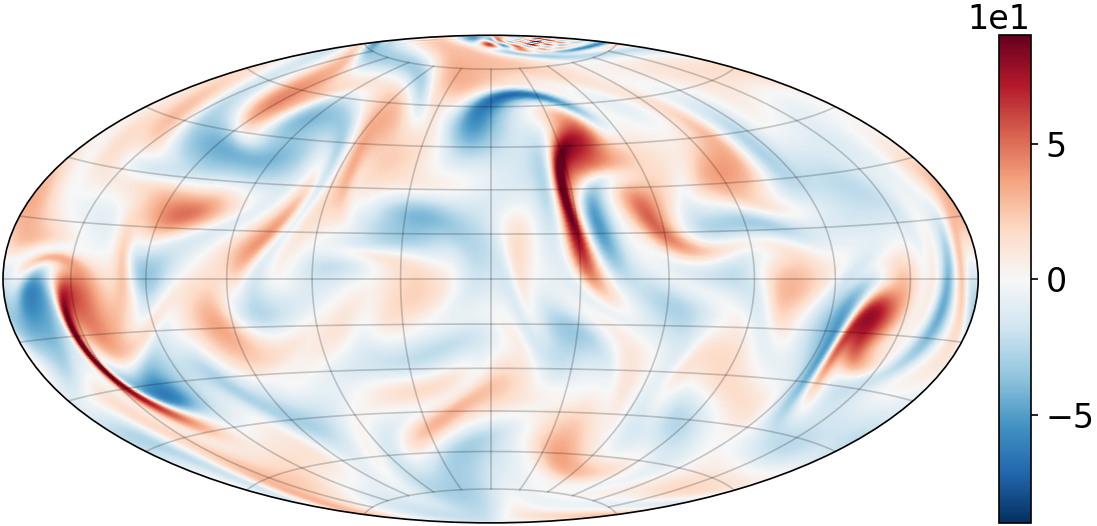}
   \end{subfigure}
   \begin{subfigure}{0.49\textwidth}
      \caption*{$\sigmamat$ at $t=1$} 
      \includegraphics[width=\textwidth]{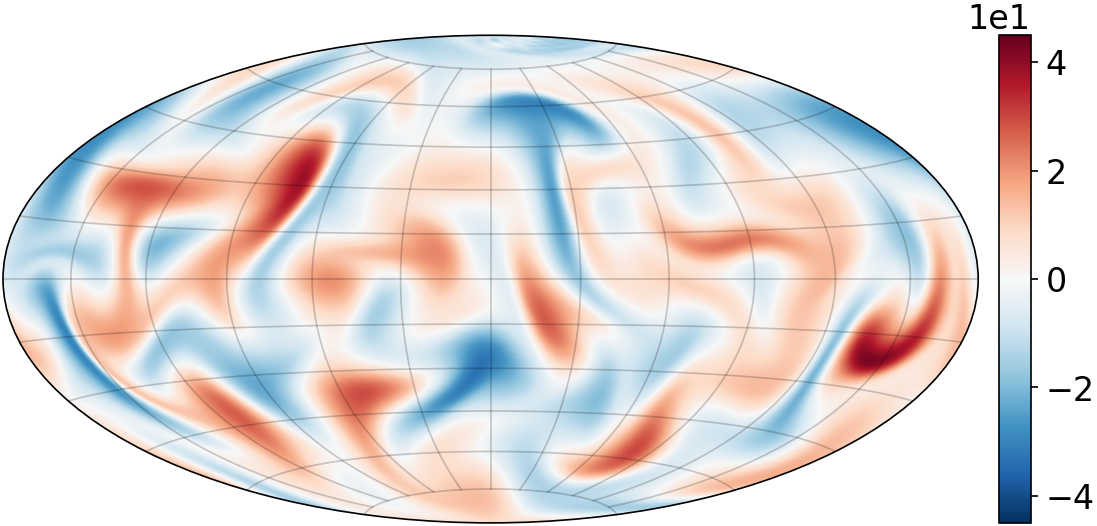}
   \end{subfigure}
   \\
   \begin{subfigure}{0.49\textwidth}
      \caption*{$\Delta_n P$ at $t=5$} 
      \includegraphics[width=\textwidth]{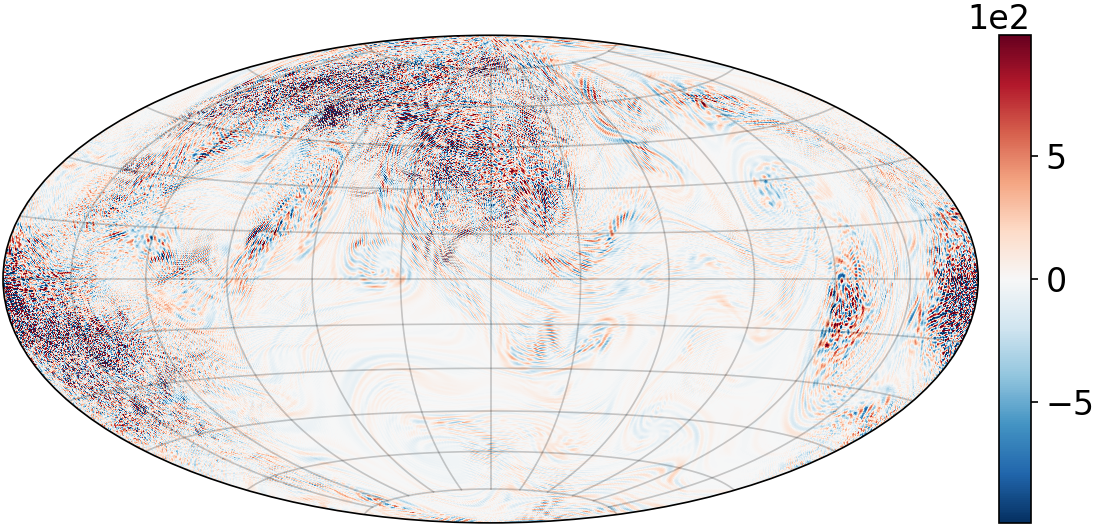}
   \end{subfigure}
   \hspace{-0.6em}
   \begin{subfigure}{0.49\textwidth}
      \caption*{$\sigmamat$ at $t=5$} 
      \includegraphics[width=\textwidth]{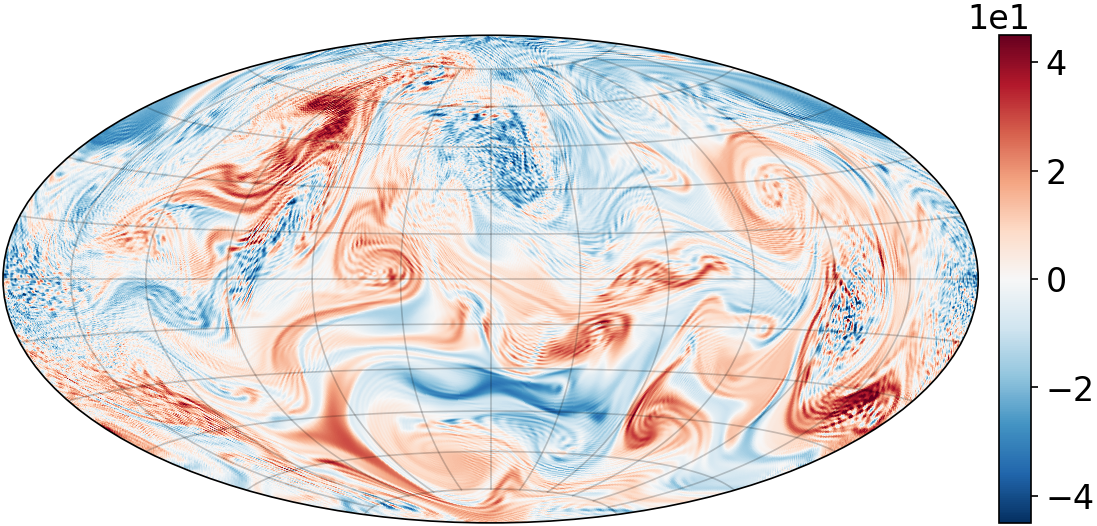}
   \end{subfigure}
   \\
   \begin{subfigure}{0.49\textwidth}
      \caption*{$\Delta_n P$ at $t=20$} 
      \includegraphics[width=\textwidth]{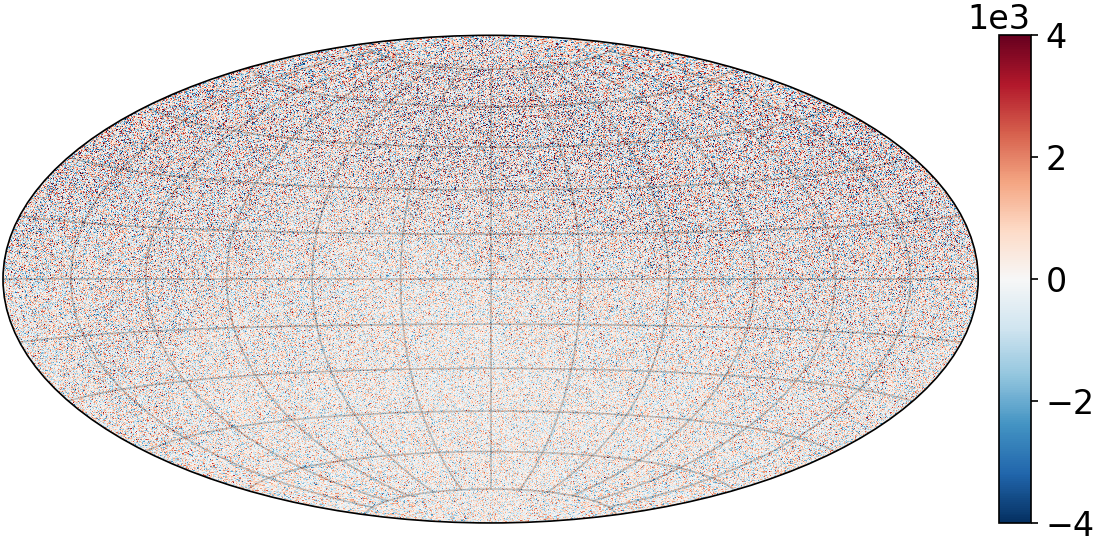}
   \end{subfigure}
   \begin{subfigure}{0.49\textwidth}
      \caption*{$\sigmamat$ at $t=20$} 
      \includegraphics[width=\textwidth]{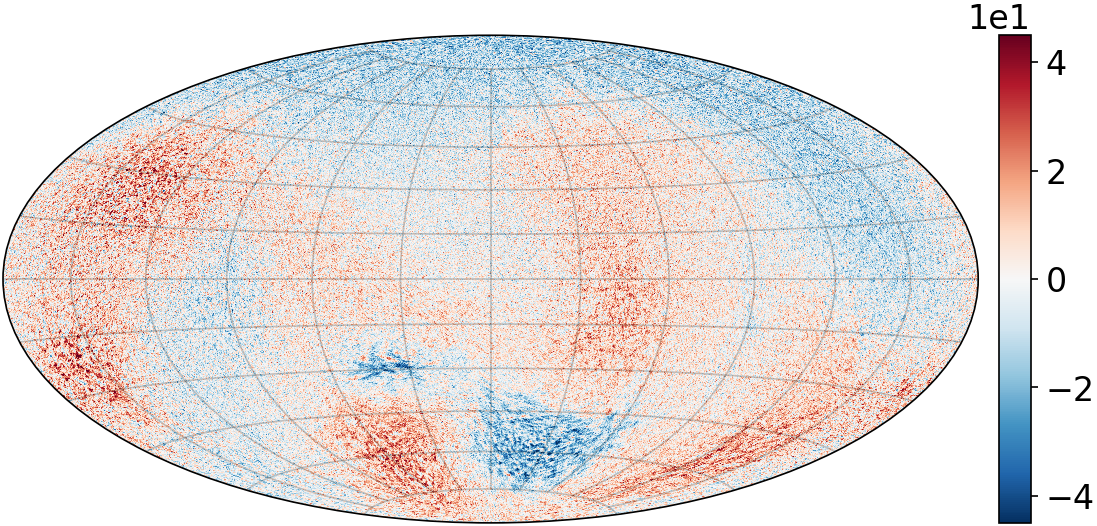}
   \end{subfigure}
   \caption{Second simulation with $n=1024$.
      Visualization of the evolution of $\Delta_nP(t)$ and $\sigmamat(t)$.
      Contrary to the 2-D Euler equations, there is no inverse energy cascade.
      In particular, the large scale structure of $\Delta_n P$ disperse into small scales.
   }
   \label{fig:random-evolution}
\end{figure}

\begin{figure}
   \centering
   \begin{subfigure}{0.8\textwidth}
      \caption*{Evolution of the vorticity supremum norm for different $n$}
      \includegraphics[width=\textwidth]{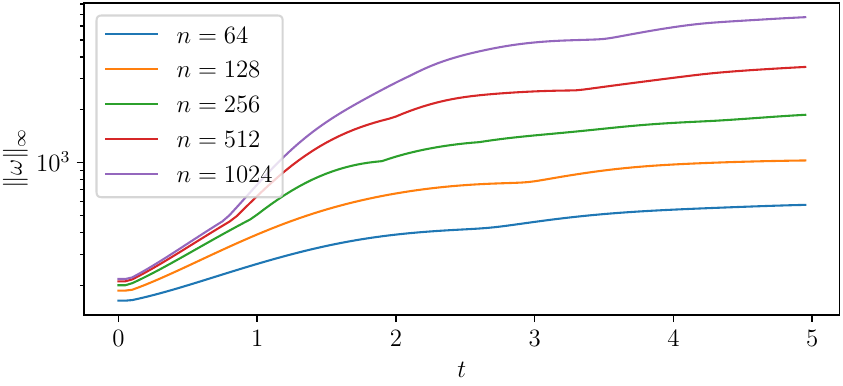}
   \end{subfigure}
   \caption{
      Second simulation.
      The supremum norm of the 3-D vorticity $\omega$ for different choices of $n$.
      Initially it grows exponentially, or faster. 
      But eventually, due to the finite $n$, it flattens out.
   }
   \label{fig:random-Linf-norm}  
\end{figure}

\new{
\section{Outlook}

The numerical experiments in section \ref{sec:numerics} showcase how the extended Zeitlin model developed in this paper can be used to yield insights and suggestions on qualitative behavior, by varying the matrix size $n$ and observe the generic trend. 
For example, Figures~\ref{fig:blowup-Linf-norm} and \ref{fig:random-Linf-norm} suggests that the $L^\infty$-norm of the vorticity grows exponentially, or slightly faster, in generic solutions.
This might give insights on whether the theoretical bound of double exponential growth is attainable on a compact domain without boundary.
(On a disk, it is attainable \cite{KiSv2014}.)
Of course, the numerical simulations do not bring us closer to proving such a thesis, but they do guide our intuition, and they could point out which directions are worthwhile to pursue.
More detailed and systematic numerical experiments, for example to find possible blow-up mechanisms, remain to be done.

It is also natural to further investigate the geometric properties of the model.
For example, how does the curvature change as $n$ increases? 
Is it always more negative than the two-dimensional case? 
Can one see lack of Fredholmness (\emph{cf.}~\cite{LiMiPr2022}) gradually appearing as $n$ goes to infinity? 
Are curvature bounds visible in the growth of vorticity, as they should be (since vorticity is a Jacobi field)?
These are deep geometrical questions that shed light on the qualitative dynamics and might be approachable in finite dimensions.

In summary, our general aim is that the extended Zeitlin model presented here should provide a useful tool to guide our understanding of the 3-D Euler equations.

}

\appendix

\section{Curvature and exact solutions when \texorpdfstring{$n=2$}{n = 2}}

\subsection{Ricci curvature}

Here we will compute the sectional curvature and the Ricci curvature for the  Zeitlin model on $\mathfrak{su}(n)\times \mathfrak{u}(n)$.
Already in the simplest possible case where $n=2$, this is surprisingly nontrivial.
For Zeitlin's model on $S^2$, the metric~\eqref{zeitlin2dmetric} then reduces to a multiple of the bi-invariant metric on $\mathfrak{su}(2)$, and the sectional curvature ends up being a positive constant 
(corresponding to the well-known identification between $SU(2)$ and the round $3$-sphere). 
However, even though our metric \eqref{zeitlinprodmet} on $\mathfrak{su}(2)\times \mathfrak{u}(2)$ restricts to multiples of the bi-invariant metric on each factor, the curvature  takes on both signs due to the nontrivial twisting involved in the product structure given by Theorem \ref{productzeitlinthm}.

\begin{theorem}\label{riccisu2thm}
If $Z = (X,cI+Y)\in \mathfrak{su}(2)\times \mathfrak{u}(2)$ for $X,Y\in \mathfrak{su}(2)$ and $c\in \mathbb{R}$, then the Ricci curvature for the metric \eqref{zeitlinprodmet} is
\begin{equation}\label{riccisu2}
\Ric(Z,Z) = \frac{1}{4}\, \lvert Y+2X\rvert^2 - \frac{11}{8} \, \lvert X\rvert^2.
\end{equation}
In particular the Ricci curvature takes on both signs. 
\end{theorem}

\begin{proof}
We first note that the Ricci curvature along the center $cI$ of $\mathfrak{u}(n)$ vanishes as it commutes with everything else.
We therefore can assume $c=0$. 

On $\mathfrak{su}(2)$, the Lie bracket in the basis $S_1,S_2,S_3$ can be identified with the usual cross product in $\mathbb{R}^3$, treating 
$(S_1,S_2,S_3$ as an oriented orthonormal basis. Hence we have $\ad_XY = -X\times Y$ for $X,Y\in \mathfrak{su}(2)$. 
Since $S_1$, $S_2$, and $S_3$ are all eigenvectors of the Hoppe-Yau Laplacian \eqref{hoppeyau} with eigenvalue $-2$, we can simply replace $\Delta_2 = -2I$, which simplifies
things greatly. 
In particular the metric \eqref{zeitlinprodmet} on the product $\mathfrak{su}(2)\times \mathfrak{su}(2)$ becomes 
$$ \big\langle (Y_1,X_1), (Y_2,X_2)\big\rangle = Q(Y_1,Y_2) + 2 Q(X_1,X_2),$$
in terms of the bi-invariant metric $Q(X_1,X_2) = - \operatorname{tr}(X_1X_2)$ on $\mathfrak{su}(2)$. 

We first compute $\adstar$. Recalling from \eqref{adjointproduct} that 
\begin{equation}\label{adjointsimplest} 
\ad_{(Y_1,X_1)}(Y_3,X_3) = \big( -X_1\times Y_3 - Y_1\times X_3 + X_1\times X_3, -X_1\times X_3\big),
\end{equation}
we find that 
\begin{align*}  
\big\langle \adstar_{(Y_1,X_1)}(Y_2,X_2), (Y_3,X_3)\big\rangle 
&= \big\langle (Y_2,X_2), \ad_{(Y_1,X_1)}(Y_3,X_3)\big\rangle \\
&= \big\langle (Y_2,X_2), ( -X_1\times Y_3 - Y_1\times X_3 + X_1\times X_3, -X_1\times X_3 )\big\rangle \\
&= \langle Y_2, -X_1\times Y_3 - Y_1\times X_3 + X_1\times X_3\rangle - 2\langle X_2, X_1\times X_3\rangle \\
&= \langle Y_3, X_1\times Y_2\rangle + \langle X_3, Y_1\times Y_2 - X_1\times Y_2 + 2 X_1\times X_2\rangle.
\end{align*}
Since this equation is valid for every $(Y_3,X_3)$, we have 
\begin{equation}\label{adstarsimplest}
\adstar_{(Y_1,X_1)}(Y_2,X_2) = \big( X_1\times Y_2, \tfrac{1}{2}(Y_1\times Y_2 - X_1\times Y_2) + X_1\times X_2\big).
\end{equation}

The terms in the curvature tensor \eqref{curvature} then take the form, with $V = (Y_2,X_2)$ and $U=(Y_1,X_1)$: 
\begin{align*}
\ad_UV &=  \big( -X_1\times Y_2 - Y_1\times X_2 + X_1\times X_2, -X_1\times X_2\big) \\
\adstar_UV+\ad_UV &= \big( X_1\times X_2 - Y_1\times X_2, \tfrac{1}{2}(Y_1\times Y_2-X_1\times Y_2)\big) \\
\adstar_UV + \ad_UV + \adstar_VU &= \big( X_1\times X_2-2Y_1\times X_2, \tfrac{1}{2}(Y_1\times X_2-X_1\times Y_2) - X_1\times X_2\big) \\
\adstar_UU &= (X_1\times Y_1, -\tfrac{1}{2} X_1\times Y_1), \qquad \adstar_VV = (X_2\times Y_2, -\tfrac{1}{2} X_2\times Y_2).
\end{align*}
Plugging in and simplifying, we obtain 
\begin{multline}\label{curvaturesimplest}
\langle R(U,V)V,U\rangle = 
\tfrac{1}{4} \big\lvert (X_1-2Y_1)\times X_2\big\rvert^2 + \tfrac{1}{8} \big\lvert Y_1\times X_2-X_1\times Y_2 - 2X_1\times X_2\big\rvert^2 \\
- \langle (X_1-Y_1)\times X_2, (X_1-Y_1)\times X_2 -X_1\times Y_2\rangle + \langle (Y_1-X_1)\times Y_2, X_1\times X_2\rangle \\
- \tfrac{3}{2} \langle  X_1\times Y_1, X_2\times Y_2\rangle.
\end{multline}

Now we consider an orthogonal basis for $\mathfrak{su}(2)\times \mathfrak{su}(2)$, given by 
$$ F_1 = (S_1,0), \quad F_2=(S_2,0), \quad F_3=(S_3,0), \quad F_4=(0,S_1), \quad F_5=(0,S_2), \quad F_6=(0,S_3).$$
Note that $\langle F_i,F_j\rangle = \delta_{ij}$ and $\langle F_{i+3},F_{j+3}\rangle = 2\delta_{ij}$ for $1\le i,j\le 3$. 
The formula \eqref{curvaturesimplest} simplifies in the case $U=F_i=(S_i,0)$ (with $X_1=S_i$ and $Y_1=0$) to 
\begin{equation}\label{Ficurvatures}
\begin{split}
\langle R(F_i,V)V,F_i\rangle &= 
\tfrac{1}{4} \lvert X_1\times X_2\rvert^2 + \tfrac{1}{8} \big\lvert X_1\times (Y_2+2X_2) \big\rvert^2 \\
&\qquad\qquad - \langle X_1\times X_2, X_1\times (X_2-Y_2)\rangle - \langle X_1\times Y_2, X_1\times X_2\rangle \\
&= -\tfrac{3}{4} \lvert S_i\times X_2\rvert^2 + \tfrac{1}{8} \big\lvert S_i\times (Y_2+2X_2) \big\rvert^2.
\end{split}
\end{equation}

Meanwhile if $U=F_{i+3}$ (with $X_1=0$ and $Y_1=S_i$), formula \eqref{curvaturesimplest} simplifies to
\begin{equation}\label{Fip3curvatures}
\begin{split}
\langle R(F_{i+3},V)V,F_{i+3}\rangle
&= 
\tfrac{1}{4} \big\lvert 2Y_1\times X_2\big\rvert^2 + \tfrac{1}{8}  \lvert Y_1\times X_2 \rvert^2 
- \langle Y_1\times X_2, Y_1\times X_2 \rangle  \\
&= \tfrac{1}{8} \lvert S_i\times X_2\rvert^2.
\end{split}
\end{equation}  

To get the Ricci curvature, we sum the expressions in \eqref{Ficurvatures}--\eqref{Fip3curvatures}
over $S_i$ for $1\le i\le 3$, taking half the sum of \eqref{Fip3curvatures} because $\langle F_{i+3},F_{j+3}\rangle = 2\delta_{ij}$ for $1\le i\le 3$. Thus we have that
$$
\Ric(Z,Z) = -\tfrac{3}{2} \lvert X\rvert^2 + \tfrac{1}{4} \lvert Y+2X\rvert^2 + \tfrac{1}{8} \lvert X\rvert^2.
$$ 
Here we used the formula 
$$ \sum_{i=1}^3 \lvert e_i\times X\rvert^2 = 2\lvert X\rvert^2$$
for the ordinary cross product in three dimensions, and replaced $(Y_2,X_2)$ with $(Y,X)$ to simplify notation. This reduces to \eqref{riccisu2}.
\end{proof}

As a consequence, we quickly find both signs of sectional curvature in the three-dimensional model, even in the simplest case. 
Meanwhile, in the 2-D Zeitlin model, small values of $n$ lead to strictly positive sectional curvature, and only higher values yield the negative curvature which is fairly common in $\Diffmu(S^2)$. 

In finite dimensions the Ricci curvature makes sense and often leads to much simpler formulas than the full Riemann curvature tensor, since it distills information into fewer dimensions. 
In infinite dimensions (on the full groups $\Diffmu(S^2)$ or $\Diffmu(S^3)$) the Ricci curvature doesn't make sense, except perhaps in an averaged sense (Lukatskii computed a version of Ricci curvature for $\Diffmu(\mathbb{T}^2)$ for example by taking averages of sectional curvatures in simple directions~\cite{lukatskii1984curvature}). 
It would be interesting to see, for each $n$, how much positive versus negative Ricci curvature we have, e.g., to find the index of the Ricci bilinear form in general. 
Here, when $n=2$, we have a $7$-dimensional configuration space, and we found that the index is $0$ (three positive eigenvalues of the Ricci tensor, three negative, and one zero in the $cI$ direction on $\mathfrak{u}(2)$). 
Is this also true for general $n$? 

\subsection{Exact solutions}

Using the formula for $\adstar$ in \eqref{adstarsimplest}, we can write down the Euler-Arnold equation on $\mathfrak{su}(2)\times \mathfrak{u}(2)$ and solve it explicitly.
Obviously, one should not expect such a solution formula for arbitrary $n$, but for $n=2$ there are a number of cancellations. 

\begin{theorem}\label{explicitsolution3d}
For $n=2$, with the Lie bracket identified with the cross product, the Euler equation \eqref{zeitlinproduct}
for $(P,\sigmamat)$ takes the form 
\begin{equation}\label{trivialequations}
\dot\sigmamat(t) + \frac{1}{\hbarN} [P, \sigmamat] = 0, \qquad P'(t) - \frac{1}{2\hbarN} [P,\sigmamat] = 0,
\end{equation}
and all solutions take the form 
$$ \sigmamat(t) = e^{-t \ad_L} \sigmamat(0), \qquad P(t) = e^{-2t \ad_L} P(0),$$
where $L = \frac{1}{\hbarN} P(0)+\frac{1}{2\hbarN} \sigmamat(0)$. 
\end{theorem}

\begin{proof}
The first equation is the same, and the second comes from the fact that on $\mathfrak{su}(2)$ we have $\Delta_2 = -2I$. 
Thus, for any solution, we must have that $P(t)+\tfrac{1}{2}\sigmamat(t)$ is constant. Call this constant matrix $\hbarN L$; then 
we have 
$$[P(t),\sigmamat(t)] = [\hbarN L-\tfrac{1}{2} \sigmamat(t), \sigmamat(t)] = \hbarN [L,\sigmamat(t)]$$
and similarly $[P(t),\sigmamat(t)] = -2\hbarN [L,P(t)]$. 
Equations \eqref{trivialequations} become
$$ \sigmamat'(t) = -\ad_L \sigmamat(t), \qquad P'(t) = 2 \ad_L P(t),$$
and the solution is immediate. 
\end{proof}

It would be interesting to see if these simple time-dependent solutions have analogues as exact, nonsteady solutions of the full axisymmetric Euler equations, along the lines of 2-D Rossby-Haurwitz waves~\cite{bennrossby}, which also survives in the Zeitlin model~\cite{Vi2020}.

\arxivonly{
\section{Explicit computation of the descending metric}\label{descendmetricappendix}

For reference, we give here explicit calculations of the expression of the descending metric \eqref{L2kineticS3}.
These calculations are adaptations of similar calculations in the paper \cite{LiMiPr2022}.

Choose Hopf-like coordinates $(r,\theta,\psi)$ for the $3$-sphere such that 
\begin{alignat*}{3}
    w&=\cos{\tfrac{r}{2}}\cos{\tfrac{\theta}{2}}, &\qquad\qquad x &= \cos{\tfrac{r}{2}}\sin{\tfrac{\theta}{2}}, \\
    y&=\sin{\tfrac{r}{2}} \cos{(\psi+\tfrac{\theta}{2})}, & 
    z&= \sin{\tfrac{r}{2}} \sin{(\psi+\tfrac{\theta}{2})}.
\end{alignat*}
Here the acceptable domain is $0<r<\pi$, $0<\theta<4\pi$, and $0<\psi<2\pi$ in order to capture almost all of the $3$-sphere.
Choose standard spherical coordinates on $S^2$ with 
$$(t,u,v) = (\sin{\rho}\cos{\phi}, \sin{\rho}\sin{\phi}, \cos{\rho}).$$
Then one can easily compute that the projection map $\Pi$ from \eqref{hopffibration} is given in these coordinates by 
$$ \rho = r, \qquad \phi = \psi.$$
Note that the usual spherical coordinate domain of $\rho$ and $\psi$ is completely covered by our domain for the $3$-sphere coordinates, with $\theta$ ranging freely over $[0,4\pi]$ corresponding to each level curve of $E_1$ having length $4\pi$.

In these Hopf-like coordinates on $S^3$, our vector fields $E_1$, $E_2$, and $E_3$ defined above take the explicit form 
\begin{align*}
    E_1 &= \partial_{\theta}\\
    E_2 &= 
    \cos{(\psi+\theta)} \partial_r - \sin{(\psi+\theta)} (\csc{r}\,\partial_{\psi}-\tan{\tfrac{r}{2}} \, \partial_{\theta})  \\
    E_3 &= 
    \sin{(\psi+\theta)} \partial_r + \cos{(\psi+\theta)} (\csc{r}\,\partial_{\psi}-\tan{\tfrac{r}{2}} \, \partial_{\theta})  
\end{align*}
and the dual basis of $1$-forms is given by 
\begin{align*}
    \alpha^1 &= d\theta + 2\sin^2{\tfrac{r}{2}} \, d\psi \\
    \alpha^2 &= \cos{(\psi+\theta)} \, dr - \sin{r} \sin{(\psi+\theta)} \, d\psi \\
    \alpha^3 &= \sin{(\psi+\theta)} \, dr + \sin{r} \cos{(\psi+\theta)} \, d\psi.
    \end{align*}
These are all declared to be orthonormal, so the volume form on $S^3$ is 
$$ dV = \alpha^1 \wedge \alpha^2 \wedge \alpha^3 = \sin{r} \, dr\wedge d\theta\wedge d\psi.$$

Functions that are invariant under $E_1$ on $S^3$ take the form 
$$ \tilde{\sigma}(r,\theta,\psi) = \sigma(r,\psi).$$
So the integrals we are dealing with look like 
$$ \int_{S^3} \tilde{\sigma}^2 \, dV 
= \int_0^{\pi} \int_0^{2\pi} \int_0^{4\pi} \sigma(r,\psi)^2 \, \sin{r} \, d\theta\,d\psi\,dr = 4\pi \int_0^{\pi} \int_0^{2\pi} \sigma(r,\psi)^2 \, \sin{r} \, d\psi \, dr.$$
Thus, the correct multiplier is $4\pi$. 

To check the Laplacian formula, write $\tilde{\sigma}(r,\theta,\psi) = \sigma(r,\psi)=\sigma(\rho,\phi)$. We compute that 
\begin{align*}
    E_2(\tilde{\sigma})(r,\theta,\psi) &= 
        \cos{(\psi+\theta)} \sigma_{\rho}(r,\psi) - \frac{\sin{(\psi+\theta)}}{\sin{r}} \, \sigma_{\phi}(r,\psi) \\
    E_3(\tilde{\sigma})(r,\theta,\psi) &= 
            \sin{(\psi+\theta)} \sigma_{\rho}(r,\psi) + \frac{\cos{(\psi+\theta)}}{\sin{r}} \, \sigma_{\phi}(r,\psi).
\end{align*}
Then applying $E_2$ and $E_3$ again, we get 
\begin{align*}
    (E_2)^2(\tilde{\sigma})(r,\theta,\psi) 
    &= \cos^2{(\psi+\theta)}\,\sigma_{\rho\rho}(r,\theta+\psi)
    + \sin^2{(\psi+\theta)} \left( \cot{r} \, \sigma_{\rho}(r,\psi)
    + \csc^2{r} \, \sigma_{\phi\phi}(r,\psi) \right)\\
    &\qquad + \sin{(2\psi+2\theta)} \csc{r} \, \big( \cot{r} \,\sigma_{\rho}(r,\psi) - \sigma_{\rho\phi}(r,\psi)\big) \\
    (E_3)^2(\tilde{\sigma})(r,\theta,\psi) 
    &= \sin^2{(\psi+\theta)}\,\sigma_{\rho\rho}(r,\theta+\psi)
    + \cos^2{(\psi+\theta)} \left( \cot{r} \, \sigma_{\rho}(r,\psi)
    + \csc^2{r} \, \sigma_{\phi\phi}(r,\psi) \right)\\
    &\qquad -\sin{(2\psi+2\theta)} \csc{r} \, \big( \cot{r} \,\sigma_{\rho}(r,\psi) - \sigma_{\rho\phi}(r,\psi)\big) .
\end{align*}
Adding these together, we get 
\begin{align*}
    (E_2)^2(\tilde{\sigma})(r,\theta,\psi) + (E_3)^2(\tilde{\sigma})(r,\theta,\psi) &= \sigma_{\rho\rho}(r,\psi)
+ \cot{r} \, \sigma_{\rho}(r,\psi) + \csc^2{r} \,\sigma_{\phi\phi}(r,\psi)\\ &= \Delta \sigma(r,\psi).
\end{align*}
}


\bibliographystyle{plain}
\bibliography{zeitlin-axisym.bib}

\begin{thebibliography}{10}

\bibitem{Ar1966}
V.~I. Arnold.
\newblock Sur la g\'eom\'etrie diff\'erentielle des groupes de {L}ie de dimension infinie et ses applications \`a l'hydrodynamique des fluides parfaits.
\newblock {\em Ann. Inst. Fourier (Grenoble)}, 16(fasc. 1):319--361, 1966.

\bibitem{ArKh1998}
V.~I. Arnold and B.~Khesin.
\newblock {\em Topological Methods in Hydrodynamics}.
\newblock Springer-Verlag, New York, 1998.

\bibitem{bennrossby}
J.~Benn.
\newblock Conjugate points in {$\mathcal{D}_\mu^s(S^2)$}.
\newblock {\em J. Geom. Phys.}, 170:104369, 2021.

\bibitem{ChPo2018}
L.~Charles and L.~Polterovich.
\newblock Sharp correspondence principle and quantum measurements.
\newblock {\em St. Petersburg Math. J.}, 29(1):177--207, 2018.

\bibitem{ChHo2023}
J.~Chen and T.Y. Hou.
\newblock Stable nearly self-similar blowup of the {2D} {B}oussinesq and {3D} {E}uler equations with smooth data {I}: Analysis.
\newblock 2023.

\bibitem{CiViMo2023}
P.~Cifani, M.~Viviani, and K.~Modin.
\newblock An efficient geometric method for incompressible hydrodynamics on the sphere.
\newblock {\em J. Comput. Phys.}, 473:111772, 2023.

\bibitem{El2021}
T.~M. Elgindi.
\newblock Finite-time singularity formation for {$C^{1,\alpha}$} solutions to the incompressible {E}uler equations on $\mathbb{R}^3$.
\newblock {\em Ann. Math.}, 194(3):647 -- 727, 2021.

\bibitem{Eu1757}
L.~Euler.
\newblock Principes g{\'e}n{\'e}raux de l'{\'e}tat d'{\'e}quilibre d'un fluide.
\newblock {\em Acad{\'e}mie Royale des Sciences et des Belles-Lettres de Berlin, M{\'e}moires}, 11:217--273, 1757.

\bibitem{FrCaCiGe2024}
A.~Franken, M.~Caliaro, P.~Cifani, and B.~Geurts.
\newblock Zeitlin truncation of a shallow water quasi-geostrophic model for planetary flow.
\newblock {\em J. Adv. Model. Earth Sys.}, 16(6):e2023MS003901, 2024.

\bibitem{Ho1989}
J.~Hoppe.
\newblock Diffeomorphism groups, quantization, and {$\mathrm{SU}(\infty)$}.
\newblock {\em Int. J. Modern Phys. A}, 04(19):5235--5248, 1989.

\bibitem{HoYa1998}
J.~Hoppe and S.-T. Yau.
\newblock Some properties of matrix harmonics on {$S^2$}.
\newblock {\em Comm. Math. Phys.}, 195(1):67--77, 1998.

\bibitem{KhLeMiPr2011b}
B.~Khesin, J.~Lenells, G.~Misio{\l}ek, and S.~C. Preston.
\newblock Curvatures of {S}obolev metrics on diffeomorphism groups.
\newblock 09 2011.

\bibitem{KiSv2014}
A.~Kiselev and V.~{\v{S}}ver{\'a}k.
\newblock Small scale creation for solutions of the incompressible two-dimensional {E}uler equation.
\newblock {\em Ann. of Math.}, 180(3):1205--1220, 2014.

\bibitem{LiMiPr2022}
L.~Lichtenfelz, G.~Misio{\l}ek, and S.~C. Preston.
\newblock Axisymmetric diffeomorphisms and ideal fluids on {R}iemannian 3-manifolds.
\newblock {\em Int. Math. Res. Not.}, 2022(1):446--485, 2022.

\bibitem{lukatskii1984curvature}
A.~M. Lukatskii.
\newblock Curvature of the group of measure-preserving diffeomorphisms of the n-dimensional torus.
\newblock {\em Siberian Mathematical Journal}, 25(6):893--903, 1984.

\bibitem{LuHo2014}
G.~Luo and T.~Y. Hou.
\newblock Toward the finite-time blowup of the {3D} axisymmetric {E}uler equations: A numerical investigation.
\newblock {\em Multiscale Model. Simul.}, 12(4):1722--1776, 2014.

\bibitem{MoPe2024}
K.~Modin and M.~Perrot.
\newblock {E}ulerian and {L}agrangian stability in {Z}eitlin's model of hydrodynamics.
\newblock {\em Comm. Math. Phys.}, 2024.

\bibitem{MoRo2024}
K.~Modin and M.~Roop.
\newblock Spatio-temporal {L}ie-{P}oisson discretization for incompressible magnetohydrodynamics on the sphere.
\newblock arXiv:2311.16045, 2024.

\bibitem{MoVi2020}
K.~Modin and M.~Viviani.
\newblock A {C}asimir preserving scheme for long-time simulation of spherical ideal hydrodynamics.
\newblock {\em J. Fluid Mech.}, 884, 2020.

\bibitem{MoVi2024}
K.~Modin and M.~Viviani.
\newblock Two-dimensional fluids via matrix hydrodynamics.
\newblock arXiv:2405.14282, 2024.

\bibitem{MoGr1980}
P.~J. Morrison and J.~M. Greene.
\newblock {Noncanonical {H}amiltonian Density Formulation of Hydrodynamics and Ideal Magnetohydrodynamics}.
\newblock {\em Phys. Rev. Lett.}, 45(10):790--794, 1980.

\bibitem{firstconjugate}
S.~C. Preston.
\newblock On the volumorphism group, the first conjugate point is always the hardest.
\newblock {\em Comm. Math. Phys.}, 267:493--513, 2006.

\bibitem{VD}
S.~Vishik and F.~Dolzhanski.
\newblock {Analogs of the {E}uler-{L}agrange equations and magnetohydrodynamis equations related to Lie groups}.
\newblock {\em Sov. Math. Doklady}, 19:149--153, 1978.

\bibitem{Vi2020}
M.~Viviani.
\newblock {\em Symplectic methods for isospectral flows and 2D ideal hydrodynamics}.
\newblock PhD thesis, Chalmers University of Technology, 2020.

\bibitem{Vi2008}
C.~Vizman.
\newblock Geodesic equations on diffeomorphism groups.
\newblock {\em SIGMA Symmetry Integrability Geom. Methods Appl.}, 4:030, 2008.

\bibitem{Ze1991}
V.~Zeitlin.
\newblock Finite-mode analogs of {2D} ideal hydrodynamics: coadjoint orbits and local canonical structure.
\newblock {\em Phys. D}, 49(3):353--362, 1991.

\bibitem{Ze2004}
V.~Zeitlin.
\newblock Self-consistent finite-mode approximations for the hydrodynamics of an incompressible fluid on nonrotating and rotating spheres.
\newblock {\em Phys. Rev. Lett.}, 93:264501, 2004.

\end{thebibliography}

 
 
 

\end{document}